\let\ORIlabel\label
\let\ORIrefstepcounter\refstepcounter
\AddToHook{package/hyperref/before}{
   \let\label\ORIlabel 
   \let\refstepcounter\ORIrefstepcounter}
\documentclass[final,onefignum,onetabnum]{siamart220329}

\usepackage{amssymb,amsfonts, amsmath} %
\usepackage{bm}
\usepackage{todonotes}
\usepackage{mathtools, relsize}
\usepackage{mathrsfs}

\makeatletter

\usepackage{multirow}
\newcommand{\R}{\ensuremath{\mathbb{R}}}

\newcommand{\vs}[1]{\textcolor{black}{#1}}
\newcommand{\cep}[1]{\textcolor{black}{#1}}

\usepackage{array}
\newcommand{\Kr}{\mathscr K}
\newcommand{\bX}{\bm X}
\newcommand{\bR}{\bm R}

\newcommand{\bA}{\bm A}

\newcommand{\bB}{\bm B}

\newcommand{\bP}{\bm P}

\newcommand{\bZ}{\bm Z}

\newcommand{\sscr}{{{\sc s}{\sc \small s}--{\sc {gcr}{\small (1)}}}}

\newcommand{\ssmr}{{{\sc s}{\sc \small s}--{\sc {mr}}}}

\newcommand{\bmat}[1]{\bm #1}
\theoremstyle{remark}
\newtheorem{remark}[theorem]{Remark}

\newtheorem{assumption}{Assumption}
\usepackage{enumitem}
\setlist[itemize]{topsep=-2pt,itemsep=0pt,leftmargin=15.0pt}

\usepackage{graphicx,epstopdf} 
\usepackage{tikz}
\usepackage{subcaption}
\usepackage{algorithm}
\usepackage{algpseudocode}
\usepackage{algorithmicx}
\usepackage{caption}

\usetikzlibrary{shapes.geometric}

\definecolor{matlabred}{rgb}{0.9047,    0.1918,    0.1988}
\definecolor{matlabblue}{rgb}{0.2941    0.5447    0.7494}
\definecolor{matlabgreen}{rgb}{	0.3718    0.7176    0.3612}
\definecolor{matlaborange}{rgb}{1.0000    0.5482    0.1000}

\usepackage{algpseudocode}
\algnewcommand{\LineComment}[1]{\Statex \(\%\) \small \textit{#1} \(\%\)}

\Crefname{ALC@unique}{Line}{Lines}
\usepackage{algorithm}

\title{A class of low-rank short recurrences for 
nonsymmetric linear matrix equations %
\thanks{Version of \today}}
\author{Davide Palitta\thanks{Dipartimento di Matematica and (AM)$^2$,
Alma Mater Studiorum Universit\`a di Bologna,
Piazza di Porta San Donato  5, I-40127 Bologna, Italy,
{\tt \{davide.palitta,valeria.simoncini\}@unibo.it}}
 \and Catherine E. Powell\thanks{Department of Mathematics, University of Manchester, Oxford Road, Manchester M13 9PL, United
Kingdom {\tt catherine.powell@manchester.ac.uk}}
 \and Valeria Simoncini$^\dagger$\thanks{IMATI-CNR, Pavia, Italy. }}

\ifpdf
\hypersetup{ pdftitle={Subspace-conjugate residual} }
\fi

\begin{document}
\maketitle

\renewcommand{\thefootnote}{\fnsymbol{footnote}}
\maketitle \pagestyle{myheadings} \thispagestyle{plain}
\markboth{ D.\ PALITTA, C.\ E.\ POWELL, and V.\ SIMONCINI}{NEW RECURRENCES FOR NONSYMMETRIC MATRIX EQUATIONS}


\begin{abstract}
We propose a new class of short matrix recurrences for the solution of nonsymmetric linear equations of the type
$\bmat{A}_1\bmat{X}\bmat{B}_1+\ldots+\bmat{A}_p\bmat{X}\bmat{B}_p=CD^T$. \cep{Building on ideas underpinning the recently introduced subspace conjugate gradient \vs{algorithm,}
we derive low-rank short recurrences that generalize
one-dimensional subspace projection methods to the nonsymmetric {\em matrix} equation setting. To limit memory consumption and maximize computational efficiency, rank truncation strategies and modern randomization procedures are incorporated into the proposed algorithms.} Computational experiments on a benchmark problem as well as a challenging discretized mixed formulation of a diffusion equation with random inputs illustrate the potential of the proposed methodology.
\end{abstract}

\begin{keywords}
Multiterm matrix equations, low-rank approximation, nonsymmetric matrix operators, sketching strategies, stochastic Galerkin method. 
\end{keywords}

\begin{MSCcodes}
 	65F45, 65F25, 65F99
\end{MSCcodes}


\section{Introduction}
We are interested in the numerical solution of large-scale multiterm matrix equations of the form
\begin{equation}\label{eq:main}
\bmat{A}_1\bmat{X}\bmat{B}_1+\ldots+\bmat{A}_p\bmat{X}\bmat{B}_p=CD^T,
\end{equation}
where for $i=1,\ldots,p$ the coefficient matrices $\bmat{A}_i\in\R^{n_A\times n_A}$, $\bmat{B}_i\in\R^{n_B\times n_B}$ are large, sparse and \emph{nonsymmetric},
and $C\in\R^{n_A\times q}$, $D\in\R^{n_B\times q}$ are tall full-rank matrices with $q\ll n_A,n_B$. Note that the solution matrix $\bm X \in \mathbb{R}^{n_{A} \times n_{B}}$ is rectangular in general. 
We define the linear operator 
\begin{eqnarray}\label{eqn:L}
\mathcal{L}:\R^{n_A\times n_B}\rightarrow\R^{n_A\times n_B}, \quad
\mathcal{L}(\bmat{X}):=\bmat{A}_1\bmat{X}\bmat{B}_1+\ldots+\bmat{A}_p \bmat{X}\bmat{B}_p,
\end{eqnarray}
so that \eqref{eq:main} can be written as $ \mathcal{L}(\bmat{X})=CD^T$,
and we assume that \eqref{eq:main} is uniquely solvable. We are mainly interested in matrix equations with $p>2$.
For $p=1$ the problem amounts to solving
two linear systems with multiple right-hand sides, while 
for $p=2$, a generalized Sylvester equation is obtained \cite{Simoncini2016}.

The matrix equation (\ref{eq:main}) can also be written as a standard (vector) linear system, 
\vspace{-0.1in}
\begin{equation}\label{eqn:Kron}
{\cal A} x = b, \qquad
{\cal A} = \sum_{i=1}^p {\bm B}_i^T \otimes {\bm A}_i ,
\end{equation}
where $x, b$ are the vectorizations of the matrices $\bm X$, $CD^T$, 
respectively. Here, $\otimes$ stands for the Kronecker product and $\cal A$ is nonsymmetric.
Although the vector formulation (\ref{eqn:Kron}) may be more familiar, the presence of the Kronecker product makes the dimension unacceptably
large, especially if ${\bm A}_i$, ${\bm B}_i$ are large themselves. If standard iterative methods are employed to solve (\ref{eqn:Kron}),
full vectors of length $n_A n_B$ need to be stored, which may be impossible to do. Under these strong memory constraints, the matrix formulation 
can provide significant benefits, the most prominent being the fact that the low-rank structure can be preserved and exploited. If the solution $\bmat{X}$ can be well approximated by low-rank matrices, then approximations can be sought that are already in low-rank factored form, significantly lowering memory requirements. 
However, we stress that a key condition for 
making the formulation (\ref{eq:main}) appealing is that the right-hand side matrix has low rank, or else can be well approximated by a low-rank matrix. In general, without this condition it is hard to ensure that $\bm X$ will be numerically low rank. See  
\cite{Benner.Breiten.13} for pioneering results on low-rank properties of the solution matrix for problems arising in the control of dynamical systems.

In recent decades interest in matrix equations of the type (\ref{eq:main}) has grown in various scientific areas. 
Indeed,  algebraic problems in matrix form naturally arise in the discretization of partial differential equations (PDEs) with separable coefficients on polygonal domains using finite difference methods, or whenever tensor approximation spaces are adopted. The latter is the case, for instance, in isogeometric analysis \cite{Sangalli.Tani.16}, and
in certain spectral methods \cite[section 5.1.3]{CHQZ}. A setting where a multiterm matrix representation of the discretized problem is very natural is space-time formulations in which 
the left and right coefficient matrices are identified with spaces associated with the distinct variables \cite{Henningetal.22}. In this setting the matrix formulation avoids another potential downside of the Kronecker formulation in (\ref{eqn:Kron}), that is the artificial mixing of quantities that may have very different behaviors and interpretations. Along the same lines, matrix equations (\ref{eq:main}) also naturally arise in the numerical solution of certain classes of parametric PDEs (or PDEs with uncertain inputs) when these are discretized using tensor product schemes that treat the spatial and parametric variables separately; see section~\ref{sec:diffusion} for one such example.
Multiterm matrix equations also classically play a key role in the analysis of stochastic or bilinear control systems, where the structure arises naturally from the problem, without any restrictive assumptions on the form of discretization \cite[section 6.4]{BCOW.17}. They also arise in other PDE-related settings, such as PDE-constrained optimization \cite{Dolgov.Stoll.17},\cite{Buengeretal.21}. Finally, we mention that multiterm matrix equations arise in data science,
image processing, and inverse problems \cite{Zhang.Nagy.18}; see \cite{Simoncini2016} for an  overview. 

Despite the nowadays rich realm of applications of matrix equations, algorithmic developments are lagging behind. Most early contributions to the solution of (\ref{eq:main}) in its generality resort to the Kronecker form (\ref{eqn:Kron}) in some form or another, and use
the matrix structure mostly to build the preconditioner or other acceleration
devices; see, e.g.,
\cite{Sangalli.Tani.16},\cite{PalittaKuerschner2021},\cite{Palitta.Simoncini.16},\cite{Shanketal.16},%
\cite{Ullman.10},\cite{Stoll.Breiten.15},\cite{Damm.08},\cite{Henningetal.22}.  Algorithms that genuinely attack (\ref{eq:main}), especially in the nonsymmetric case, are scarce.  

Existing contributions can be divided into two main streams: projection methods and short recurrences. 
Methods in the first class can be successfully applied as long as left and right approximation
spaces can be built that contain enough spectral information relating to the matrices  $\bm{A}_i$, $\bm{B}_i$, respectively \cite{Buengeretal.21}.
On the other hand, short recurrences  implicitly build an approximation space.
Given an initial ${\bm X}_0$, such methods determine a sequence of approximations $\{{\bm X}_k\}_{k\ge 0}$
as ${\bm X}_{k+1}={\bm X}_{k} + {\bm M}_{k}$ where the update matrix ${\bm M}_{k}$ is usually forced to have low rank and kept in factored form. 
For instance, in the symmetric case, early approaches 
transformed vector methods for (\ref{eqn:Kron}), say the Conjugate Gradient (CG) method, into a matrix iteration. To maintain low-rank iterates (for the solution approximation, direction and residual matrices) rank truncation is performed; in \cite{KressnerTobler2011} the method was developed in detail, and then further used, e.g.,
in \cite{Benner.Onwunta.Stoll.15},\cite{Kressner.Plesinger.Tobler.14}.
Other perspectives for recurrences include alternating methods derived for a particular nonsymmetric PDE-constrained optimization problem in \cite{Dolgov.Stoll.17}, and optimization approaches, see, e.g., \cite{Biolietal2024} for the symmetric case and \cite{Kressner2015} as a rank-one update for the nonsymmetric case. 

In this paper we introduce a new family of methods that generalize the well established class of {\it Generalized Conjugate Residual} vector methods (which culminated in the GMRES algorithm) to the nonsymmetric problem (\ref{eq:main}). To do this, we leverage recent ideas from \cite{Palittaetal2025} for the symmetric and positive definite case. Briefly, starting from matrix-oriented CG, a recurrence of the form ${\bm X}_{k+1}={\bm X}_{k} + P_k^{(l)} {\bm \alpha}_k (P_k^{(r)})^T$ was proposed, where ${\bm \alpha}_k $ is a {\it matrix} obtained by solving a local minimization problem, 
and the pair $P_k^{(l)}, P_k^{(r)}$ generates an approximation space that is expanded
as the iteration proceeds. A truncation strategy to control the rank growth was also implemented. 
We propose a new short recurrence of the same type as in \cite{Palittaetal2025} where the matrices ${\bm \alpha}_k$, $P_k^{(l)}$, and $P_k^{(r)}$ are now selected to satisfy local optimality properties that are appropriate for the case when ${\cal L}$ is nonsymmetric. 

Classical one-dimensional projection methods for vector linear systems are first reviewed in section \ref{sec:vectors}.
The new class of methods is derived in section~\ref{sec:short_rec}.
Within this class, we derive a
 local Minimal Residual method and a  one-term Generalized Conjugate Residual iteration. In section \ref{sec:algos} we present algorithmic details and discuss mechanisms for making the new methods computationally efficient on large scale problems, including rank truncation and randomization strategies to reduce memory requirements. Preconditioning strategies are discussed in section \ref{sec:preconditioning} and convergence analysis is presented in section \ref{sec:analysis}. We illustrate the performance of our new methods first on a benchmark problem; see section~\ref{sec:benchmark}. Finally, in section~\ref{sec:diffusion} we focus on a challenging matrix equation that initially motivated this work, arising from a stochastic Galerkin discretization of a parametric PDE.

\subsection{Notation}\label{Notation}
Throughout the paper, capital bold letters ($\bX$) are used to denote matrices of large dimension, such as $n_A\times n_B$, $n_A\times n_A$, or $n_B\times n_B$, with capital letters ($X$) denoting their possibly low-rank factors, e.g. $\bX=X_1X_2^T$.  Greek letters ($\alpha$) will denote scalars whereas bold Greek letters ($\bm\alpha$)
will be used for matrices of small dimension. Hence, $\text{blkdiag}(\bm{\alpha}_1, \ldots,\bm{\alpha}_s)$ will
denote a block diagonal matrix with small matrices
$\bm{\alpha}_1, \ldots,\bm{\alpha}_s$ on its diagonal blocks.
The symbol $\otimes$ denotes the Kronecker product and $\text{vec}(\cdot)$ is the
operator that stacks the columns of a matrix one below the other to form a vector. 
For $\cal L$ defined in (\ref{eqn:L}), we also define the operator
$$
\mathcal{L}^{*}(\bmat{X}):= \bmat{A}_1^T\bmat{X}\bmat{B}_1^T+\ldots+\bmat{A}_p ^T\bmat{X}\bmat{B}_p^T.
$$
For $R\in\R^{n_A\times s}$ we use the short-hand notation 
$\bA_\star \bullet R =[\bA_1R, \ldots, \bA_p R]$, 
and analogously for $\bB_\star^T\bullet  R$ with $R\in\R^{n_B\times s}$. We say that a real nonsymmetric
matrix ${\cal A}$ is positive definite if $x^T {\cal A} x>0$ for any nonzero real vector $x$.
 
The notation $e_i$ is used for the $i$-th column of the identity matrix, whose dimension will be clear from the context, while $\mathbf{1}_n\in\mathbb{R}^n$ denotes the vector of all ones.
For a matrix $\bX$, $\text{range}(\bX)$ is the space spanned by its columns. Finally, $\|\cdot \|$ denotes the Euclidean norm for vectors and its induced norm for matrices, while $\|\cdot \|_F$ denotes the matrix Frobenius norm.

\section{Classical projection methods for linear systems}\label{sec:vectors}
In this section we recall a few classical iterative methods that are employed for solving linear systems of equations ${\mathcal A} x = b$, when ${\mathcal A}$ may be nonsymmetric. These simple solution strategies, categorized as one-dimensional projection methods, serve as our starting point for developing new matrix-oriented low-rank methods in the sequel. Given a starting approximation $x_0$ and the corresponding residual 
$r_0=b- {\mathcal A} x_{0}$, a sequence $\{x_k\}_{k\ge 0}$ of approximations is determined as
\begin{equation}\label{eqn:x_rec_MR}
    x_{k+1}= x_k + \alpha_kr_k  ,\quad 
r_{k+1}= b - {\mathcal A} x_{k+1},
\end{equation}
for some constant $\alpha_k$ whose choice  completely defines the method; see, e.g., \cite[section 5.3]{Saad2003}. We are particularly 
interested in the case where, at each iteration $k$, $\alpha_k$ is chosen so that the function
$\phi(\alpha):=\|b-{\mathcal A} (x_{k}+\alpha r_k)\|^2$ 
is minimized,
that is (see, for example, \cite[section 5.3.2]{Saad2003})
$$
\alpha_k = \frac{({\mathcal A} r_k)^T r_k}{({\mathcal A} r_k)^T{\mathcal A} r_k} .
$$
This ensures that $r_{k+1}$ is orthogonal to range$({{\mathcal A} r_k})$, that is, $r_{k+1}$ satisfies a Petrov-Galerkin condition with respect to a one-dimensional subspace. In keeping with classical literature, we shall refer to this procedure as Minimal Residual (MR) iteration.

An alternative class of approaches injects more subspace information by generating direction vectors that satisfy certain orthogonality properties. The Generalized Conjugate Residual (GCR) method falls into this class.
At each iteration, the approximation is updated as
\begin{eqnarray}\label{eqn:alpha_CR}
x_{k+1}= x_k + \alpha_k p_k , \quad 
r_{k+1}= b - {\mathcal A} x_{k+1},
\end{eqnarray}
where, starting with $p_0=r_0$, a new recurrence of ``direction'' vectors $\{p_k\}_{k\ge 0}$ is introduced. 
GCR imposes the condition that {\it all} vectors ${\mathcal A} p_k$ be orthogonal and for nonsymmetric matrices ${\mathcal A}$ this constraint needs to be imposed explicitly. To make the procedure sustainable in terms of computational cost and memory, the orthogonality condition
may be imposed only with respect to $\ell\leq k$ vectors, giving rise to the so-called {\sc orthomin}($\ell$) algorithm. Here, the search directions are updated as
\begin{equation}\label{eqn:pvec}
p_{k+1} = r_{k+1} + \sum_{j=k-\ell+1}^{k} \beta_k^{(j)} p_j, \qquad
\beta_k^{(j)} = - \frac{({\mathcal A} r_{k+1})^T{\mathcal A} p_j}{({\mathcal A} p_j)^T{\mathcal A} p_j} .
\end{equation}
We are particularly interested in the case $\ell=1$, resulting in the following {\sc orthomin}(1) method,
\begin{eqnarray}\label{orthomin_recurrence}
x_{k+1}&=& x_k + \alpha_k p_k , \qquad  \alpha_k = \frac{({\mathcal A} \cep{p}_k)^T r_k}{({\mathcal A} p_k)^T{\mathcal A} p_k} \notag\\
r_{k+1}&=& b - {\mathcal A} x_{k+1}, \qquad
p_{k+1} = r_{k+1} + \beta_k p_k, \qquad \beta_k = - \frac{({\mathcal A} r_{k+1})^T{\mathcal A} p_k}{({\mathcal A} p_k)^T{\mathcal A} p_k}.
\end{eqnarray}
This strategy provides a reasonable trade-off between storage demand and computational cost. It has lower memory requirements than GCR, while imposing stronger orthogonality conditions compared to the MR iteration. Note that if ${\mathcal A}$ were symmetric, {\sc orthomin}(1) would correspond to the Conjugate Residual (CR) method where the mutual orthogonality of all vectors $\mathcal{A}p_k$ is guaranteed by the symmetry of ${\mathcal A}$, even for $\ell=1$. We refer to~\cite{taxonomy} for a comprehensive description of the above methods and their interrelations in the context of iterative solvers for linear systems of equations.

\section{Low-rank short matrix recurrences}\label{sec:short_rec}
We are now interested in adapting the MR and {\sc orthomin}-type recurrences ((\ref{eqn:x_rec_MR}) and~\eqref{orthomin_recurrence}, respectively) to our matrix equation setting. {As already mentioned, a naive strategy would be to first transform \eqref{eq:main} into ${\cal A} x = b$ as in (\ref{eqn:Kron}) 
and then apply} an iteration of the form $x_{k+1}=x_k + \alpha_k p_k$ for a specific choice of $p_k$ {and $\alpha_k$}. This would fail to exploit both the Kronecker structure of ${\mathcal A}$ and a low-rank matrix representation of $x=\textrm{vec}(\bmat{X})$. Our new idea then is to devise a principled matrix-oriented generalization of the classical (vector) one-dimensional projection scheme (\ref{eqn:x_rec_MR}), where the low-rank representation of iterates is preserved from one step to the next. More precisely, given a starting approximation $\bmat{X}_{0}$, we aim to derive recurrences of the form
\begin{equation}\label{eqn:Xiter}
\bmat{X}_{k+1} = \bmat{X}_{k} + V_k^{(l)} \bmat{\alpha}_k (
V_k^{(r)})^T , \quad
\bmat{R}_{k+1} =C D^T -  {\cal L}(\bmat{X}_{k+1}),
\end{equation}
where the pair ($V_k^{(l)}$, $V_k^{(r)}$) replaces the direction vector and $\bmat{\alpha}_k$ is now a \emph{matrix} of conforming size. Note that in a practical implementation the approximate solution is kept in factored form, so that the update is performed accordingly (see section~\ref{sec:algos}). 
The idea of using a matrix iteration of the form (\ref{eqn:Xiter}) with $\bmat{\alpha}_k $ being a matrix was  recently introduced in \cite{Palittaetal2025} for symmetric and positive definite operators.
The use of short recurrences for nonsymmetric $\cal A$ allows us to generalize this idea to the case of the nonsymmetric operator $\cal L$, or to the case of a preconditioning strategy that makes the preconditioned operator nonsymmetric; see one such example in section~\ref{sec:diffusion}.

 \subsection{The~\ssmr\ method}\label{sec:mr}
In this section we design the matrix counterpart of the MR iteration~\eqref{eqn:x_rec_MR}.
 Let $\bmat{R}_0=CD^T - {\cal L}(\bmat{X}_0)$ and write
 $\bmat{R}_0=R_0^{(l)} (R_0^{(r)})^T$. Here and in the following we assume that $\bmat{X}_0$ is such that $\bmat{R}_0$ has low rank. {\color{black} Unless important prior information about the problem at hand is available, we recommend setting $\bmat{X}_{0}=0$ to
 \vs{avoid the early}
 rank growth of the iterates.} 
To generalize the iteration (\ref{eqn:x_rec_MR}) we consider the recurrence
 \begin{eqnarray*}
\bmat{X}_{k+1} &=& \bmat{X}_{k} + R_k^{(l)} \bmat{\alpha}_k (
R_k^{(r)})^T \\
\bmat{R}_{k+1} &=& C D^T -  {\cal L}(\bmat{X}_{k+1}), \qquad 
 \bmat{R}_{k+1}=:R_{k+1}^{(l)} (R_{k+1}^{(r)})^T ,
\end{eqnarray*}
where, at step $k$, $\bmat{\alpha}_k$ is chosen to minimize the Frobenius norm of the residual, namely 
 \begin{equation}\label{eqn:MR_matrix}
 \min_{\bm{\alpha}\in\R^{q_k\times q_k}}
 \|CD^T-\mathcal{L}(\bmat{X}_k + R_k^{(l)} \bmat{\alpha} (
R_k^{(r)})^T)\|_F^2.
 \end{equation}
 Here $q_k$ is the column dimension of the full rank matrix
 $R_{k}^{(l)}$. Equivalently, equation~\eqref{eqn:MR_matrix} corresponds to 
\begin{equation}\label{eqn:MR_matrixL}
\min_{\bm{\alpha}\in\R^{q_k\times q_k}}\left\|\bmat{R}_k-\mathcal{L}\left(R_k^{(l)}\bm{\alpha}(R_k^{(r)})^T\right)\right\|_F^2.
 \end{equation}

The following result provides the solution to this minimization problem, as a generalization of the case where $\alpha_k$ is a scalar.
\begin{proposition}\label{prop:alpha}
 Let $\bmat{R}_k=R_k^{(l)} (R_k^{(r)})^T$. The minimizer $\bm{\alpha}_k$ of~\eqref{eqn:MR_matrix} is the solution to the following reduced multiterm matrix equation
 \begin{equation}\label{eqn:projectedeq_alpha}
(R_k^{(l)})^T\mathcal{L}^* \left (\mathcal{L} (R_k^{(l)} \bm{\alpha} (R_k^{(r)})^T ) \right) R_k^{(r)} =(R_k^{(l)})^T\mathcal{L}^* (\bmat{R}_k) R_k^{(r)} .
 \end{equation}
 Moreover, 
 $ \textrm{vec}(\bmat{R}_{k+1})\perp
\mathcal{A}\cdot\textrm{Range}(R_k^{(r)}\otimes R_k^{(l)})$.
\end{proposition}
\begin{proof}
By defining $W=R_k^{(r)}\otimes R_k^{(l)}$ \cep{and $z= \textrm{vec}(\bm\alpha)$} and writing $r_k = {\rm vec}(\bmat{R}_k)$, the minimization problem (\ref{eqn:MR_matrixL}) can be recast as 
$$
\min_{\underline{\bm\alpha}\in {\mathbb R}^{q_k^2}} \| r_k - {\cal A} W \cep{z}\|^2.
$$
Hence, \cep{$z$} solves the normal equation
$ ({\cal A} W)^T ({\cal A} W)\cep{z}= ({\cal A} W)^T r_{k}$. 
Going back to matrix form, this normal equation reads as follows
\begin{equation}\label{eqn:doublesum}
\sum_{i=1}^p\sum_{j=1}^p(R_k^{(l)})^T\bmat{A}_i^T\bmat{A}_j R_k^{(l)}\bm{\alpha}(R_k^{(r)})^T\bmat{B}_j\bmat{B}_i^T R_k^{(r)}=(R_k^{(l)})^T\mathcal{L}^*(\bmat{R}_k) R_k^{(r)},
\end{equation}
which is indeed (\ref{eqn:projectedeq_alpha}).
The orthogonality condition follows from standard properties of the residual of least squares problems.
\end{proof}

The explicit form (\ref{eqn:doublesum}) of the matrix equation (\ref{eqn:projectedeq_alpha}) reveals that the coefficient operator consists of $p^2$ terms, which have to be computed at each iteration. We postpone the discussion of computational strategies to solve this equation to section~\ref{sec:alphabeta}. Proposition \ref{prop:alpha} and its proof show that the matrix iteration relies on a Petrov-Galerkin orthogonality constraint with respect to a Kronecker structured subspace, corresponding to the minimization of the residual norm. In the sequel we thus refer to this procedure as the Subspace Minimal Residual method, or \ssmr\ for short.

\subsection{The~\sscr\ method}\label{sscr}
In this section, we derive the matrix counterpart of 
the generalized conjugate residual methods that were described for the vector setting in section~\ref{sec:vectors}. The recurrence for the approximate solution can be written as in (\ref{eqn:Xiter}), where 
the pair $(V_k^{(l)}, V_k^{(r)})$ is now renamed $(P_k^{(l)}, P_k^{(r)})$ so that 
 \begin{equation*}
\bmat{X}_{k+1} = \bmat{X}_{k} + P_k^{(l)} \bmat{\alpha}_k (
P_k^{(r)})^T,\quad 
\bmat{R}_{k+1} = C D^T -  {\cal L}(\bmat{X}_{k+1}),
\end{equation*}
with $\bmat{R}_{k+1}$ retained in factored form as $\bmat{R}_{k+1}=R_{k+1}^{(l)} (R_{k+1}^{(r)})^T$.
The coefficient matrix $\bm{\alpha}_{k}$ is again obtained by minimizing the Frobenius norm of the residual, so that it now satisfies
 \begin{equation}\label{eq:projectedeq_alpha}
(P_k^{(l)})^T\mathcal{L}^*\left(\mathcal{L} (P_k^{(l)} \bm{\alpha}_{k} (P_k^{(r)})^T) \right) P_k^{(r)} =(P_k^{(l)})^T\mathcal{L}^*(\bmat{R}_k) P_k^{(r)} .
 \end{equation}

As a counterpart of (\ref{eqn:pvec}), the matrix sequence 
$\{\bmat{P}_k\}_{k\ge 0}$, with
$\bmat{P}_k = P_{k}^{(l)} (P_{k}^{(r)})^T$, could be defined using
\begin{equation}\label{eqn:Pmatj}
\bmat{P}_{k+1} = \bmat{R}_{k+1} + \sum_{j=k-\ell+1}^k P_j^{(l)} \bmat{\beta}_j (P_j^{(r)})^T,
\end{equation}
for a set of matrix-valued coefficients $\bmat{\beta}_j$ to be computed. Given the presumably very high cost of computing more than one such coefficient,  in the following we only consider the case $\ell=1$, corresponding to the vector {\sc orthomin}(1) iteration \eqref{orthomin_recurrence}. In a way, the use of a subspace-based recurrence (with a matrix-valued $\bmat{\beta}_{k}$) may be viewed as a replacement for the multiterm sum in  (\ref{eqn:Pmatj}). To simplify notation from now on, we shall refer to our recurrence as \sscr.
We thus write
$$\bmat{P}_{k+1}=\bmat{R}_{k+1}+P_k^{(l)}\bm{\beta}_k(P_k^{(r)})^T,$$
where $\bm{\beta}_k\in\R^{q_k\times q_k}$ is computed by imposing the condition that 
{$\mathcal{L}(\bmat{P}_{k+1})$ is orthogonal to $\mathcal{L}(\bmat{P}_{k})$,
or, equivalently, that $\bmat{P}_{k+1}$ 
is $\mathcal{L}^*\mathcal{L}$-orthogonal to $\bmat{P}_k$.} 
This means that we impose the condition
\begin{equation}\label{eq:L2orth_direction}
(P_k^{(l)})^T\mathcal{L}^*(\mathcal{L}\left(\bmat{P}_{k+1}\right))P_k^{(r)}=0,
\end{equation}
and a direct computation shows that this is equivalent to computing $\bm{\beta}_k$ as the solution of the following projected equation
 \begin{equation}\label{eq:projectedeq_beta}
(P_k^{(l)})^T\mathcal{L}^*\left (\mathcal{L} (P_k^{(l)} \bm{\beta}_{k} (P_k^{(r)})^T)\right) P_k^{(r)}=-(P_k^{(l)})^T\mathcal{L}^*(\mathcal{L}(\bmat{R}_{k+1}))P_k^{(r)}.
 \end{equation}
\cep{In the symmetric case (i.e., when ${\cal L}^{*} = {\cal L}$), enforcing \eqref{eq:L2orth_direction} ensures that $\bmat{P}_{k+1}$ is $\mathcal{L}^{2}$-orthogonal to $\bmat{P}_k$ and we can establish the following result. We defer discussion of the
properties of \sscr\ in the non-symmetric case to Section \ref{sec:analysis}.}
\begin{proposition}\label{Prop:Discent_direction}
{\color{black} Assume $\mathcal{L}^*=\mathcal{L}$ and} let $\bmat{P}_{k+1}= {P}_{k+1}^{(l)}  (P^{(r)}_{k+1})^{T}$. {\color{black} Assume further that ${\cal L}(\bmat{P}_{k+1}) \neq 0$. If}
$$\Phi(\bmat{X}_{k+1}):=\|CD^T-{\cal L}(\bmat{X}_{k+1})\|_F^2 =
\|\bmat{R}_k-{\cal L}(P_k^{(l)}\bmat{\alpha}_k (P_k^{(r)})^T)\|_F^2,$$
then {\color{black}$\mathcal{L}({\bP}_{k+1})$} is a descent direction for $\Phi$, that is, {\color{black}$\langle \nabla\Phi(\bmat{X}_{k+1}), \mathcal{L}({\bP}_{k+1})\rangle_F < 0$}.
\end{proposition}

\begin{proof}

Let ${\cal A}$ and $b$ be as in (\ref{eqn:Kron}) {\color{black} and note that ${\cal A}$ is symmetric under the stated assumption.} Now define $\phi(x)= \|b-{\cal A}x\|^2$  and  notice that
$\phi(x)=\Phi(\bmat{X})$ where $x={\rm vec}(\bmat{X})$. It follows that {\color{black}$\nabla\phi(x)= -2 {\cal A} (b-{\cal A}x)$} and so
{\color{black}$\nabla\Phi(\bmat{X}_{k+1}) =  -2\mathcal{L}(\bmat{R}_{k+1})$}, where
$\bmat{R}_{k+1}= CD^T - {\cal L}(\bmat{X}_{k+1})$. Hence,
{\color{black}
\begin{equation*}
\begin{split}
\langle \nabla\Phi(\bmat{X}_{k+1}), \mathcal{L}(\bmat{P}_{k+1})\rangle_F &
= -2\langle \mathcal{L}(\bmat{R}_{k+1}),  \mathcal{L}(\bmat{P}_{k+1})\rangle_F\\
&= -2\langle \mathcal{L}(\bmat{P}_{k+1})-
\mathcal{L}({P}_{k}^{(l)}\bmat{\beta}_k({P}_k^{(r)})^{T})
,  \mathcal{L}(\bmat{P}_{k+1})\rangle_F\\
& = -2 \|\mathcal{L}(\bmat{P}_{k+1})\|^2_F +2\langle \mathcal{L}({P}_{k}^{(l)}\bm{\beta}_k({P}_k^{(r)})^{T}),\mathcal{L}(\bmat{P}_{k+1})
\rangle_F\\
&= - 2\|\mathcal{L}(\bmat{P}_{k+1})\|_F^2
+2\,\text{trace}\left(\bmat{\beta}_k^T({P}_k^{(l)})^{T}
\mathcal{L}(\mathcal{L}(\bmat{P}_{k+1}))
{P}_k^{(r)}\right).
\end{split}
\end{equation*}}
{\color{black}The $\mathcal{L}^2$-orthogonality
of the matricized directions (cf.~\eqref{eq:L2orth_direction} with $\mathcal{L}^*=\mathcal{L}$)} then gives the stated result. 
\end{proof}

 \begin{algorithm}[t]
{\footnotesize
\begin{algorithmic}[1]
\smallskip
\Statex \textbf{Input:} Operator $\mathcal L:\mathbb{R}^{n_A\times n_B}\rightarrow \mathbb{R}^{n_A\times n_B}$ in~\eqref{eq:main}, low-rank factors $C$, $D$ of the right-hand side, initial guess $\bX_0$, maximum number of iterations $\texttt{maxit}$, stopping tolerance $\texttt{tol}$, truncation parameters {\tt maxrank} and {\tt toltrank}$^{(\dagger)}$. 
\Statex \textbf{Output:} Approximate solution \cep{$\bX_{k+1}$}  such that $\|\mathcal L(\cep{\bX_{k+1}})- C D^{T} \|_{F}\leq \| \cep{\bmat{R}_{0}} \|_{F} \cdot \texttt{tol}$
\smallskip

\State Set $\bR_0=C D^{T}-\mathcal{L}(\bX_0)$, $\bP_0=\bR_0=P_0^{(l)}(P_0^{(r)})^{T}$
\For{$k=0,\ldots,\mathtt{maxit}$}
\State Compute $\bm{\alpha}_k$ by solving~\eqref{eq:projectedeq_alpha}\label{line:compute_alpha}
\State Set $\bX_{k+1}=\bX_k+P_k^{(l)}\bm{\alpha}_k (P^{(r)}_k)^T$ in a factorized form $X_{k+1}^{(l)}\bm{\tau}_{k+1}(X_{k+1}^{(r)})^{T}=\bX_{k+1}$
 {\flushright{\hspace{7cm} optional: low-rank truncation of $\bX_{k+1}$}}
\State Set $\bR_{k+1}=CD^{T}-\mathcal{L}(X_{k+1}^{(l)}\bm{\tau}_{k+1}(X_{k+1}^{(r)})^{T})$ in a factorized form $R_{k+1}^{(l)}\bm{\rho}_{k+1}(R_{k+1}^{(r)})^{T}=\bR_{k+1}$
 {\flushright{\hspace{7cm} optional: low-rank truncation of $\bR_{k+1}$}}
\If{$\|\bR_{k+1}\|_{F}\leq \|{\color{black}\bR_{0}}\|_{F} \cdot \texttt{tol}$}
\State Return $\bX_{k+1}$
\EndIf
\State Compute $\bm{\beta}_k$ by solving~\eqref{eq:projectedeq_beta}\label{line:compute_beta}
\State Set $\bP_{k+1}=\bR_{k+1}+P_k^{(l)}\bm{\beta}_k (P_k^{(r)})^{T}$ in a factorized form $P_{k+1}^{(l)}\bm{\gamma}_{k+1}(P_{k+1}^{(r)})^{T}=\bP_{k+1}$
 {\flushright{\hspace{7cm} optional: low-rank truncation of $\bP_{k+1}$}}\label{alg:linePk1}
\EndFor
\State Return $\bX_{k+1}$
\end{algorithmic}    \caption{\sscr\ - Subspace Generalized Conjugate Residual method.
\label{alg:subspaceCR} }
}
{\scriptsize$(\dagger)$: Parameters used in the low-rank truncation procedures at step 4, 5, and 10.}
\end{algorithm}

\section{The Algorithms}\label{sec:algos}
The \sscr\ scheme for \eqref{eq:main} is summarized in Algorithm~\ref{alg:subspaceCR}.  The \ssmr\ algorithm is obtained
by simply replacing lines 9--10 with 
$$
{\rm Set}\qquad \bmat{P}_{k+1}=\bmat{R}_{k+1},
$$
which corresponds to replacing $\bmat{P}_{k}$ with $\bmat{R}_k$ throughout the algorithm. The low-rank truncation of $\bmat{R}_{k+1}$ is then carried on. Note that $\bm{\beta}_k$ does not need to be computed in this case. This makes the cost per iteration of \ssmr\ lower than that of \sscr, especially when the rank of the factors is large. Moreover, no extra memory allocation for the matrices $\bmat{P}_k$ is required, making the method very appealing memory-wise. On the other hand, the weaker orthogonality condition imposed by \ssmr\ may result in a higher number of iterations being needed to meet a prescribed accuracy compared to \sscr.

Algorithm~\ref{alg:subspaceCR} includes some optional truncation steps.
As previously mentioned, the large dense matrices $\bX_{k+1}, \bP_{k+1}$ and $\bR_{k+1}$ are not explicitly formed.
Instead, each matrix is kept in factored form, and their ranks are truncated if necessary. The update of the approximate solution is linked to the subsequent factorization step as follows. Starting from
$\bX_{k}= X_{k}^{(l)} \bm{\tau}_{k} (X_{k}^{(r)})^T$, we have
\begin{eqnarray*}
\bX_{k+1} &= & \bX_k + P_k^{(l)} \bm{\alpha}_k (P_k^{(r)})^T \\
& = &  [X_k^{(l)}, P_k^{(l)} ] {\rm blkdiag}( \bm{\tau}_k, \bm{\alpha}_k)
 [X_k^{(r)}, P_k^{(r)} ]^T  =  
 X_{k+1}^{(l)} \bm{\tau}_{k+1} (X_{k+1}^{(r)})^T,
\end{eqnarray*}
where $X_{k+1}^{(l)}$ and $X_{k+1}^{(r)}$ are the reduced orthonormal factors of the QR decompositions of
$[X_k^{(l)}, P_k^{(l)} ]$ and $[X_k^{(r)}, P_k^{(r)} ]$, respectively. That is\footnote{\vs{For completeness, we notice that one may replace the reduced QR with a rank-revealing QR or an SVD, to more easily detect possible rank deficiency in the factorized matrices. This feature was not needed in our experiments.}}, if $[X_k^{(l)}, P_k^{(l)} ]=Q^{(l)}\bm{r}_l$ and $[X_k^{(r)}, P_k^{(r)} ]=Q^{(r)}\bm{r}_r$, then we can define $X_{k+1}^{(l)}:=Q^{(l)}$, $X_{k+1}^{(r)}:=Q^{(r)}$, and $\bm{\tau}_{k+1} :=
\bm{r}_l {\rm blkdiag}( \bm{\tau}_k, \bm{\alpha}_k) \bm{r}_r^T$. Alternatively, using a (truncated) \cep{singular value decomposition} we can lower the
rank of $X_{k+1}^{(l)}$ and $X_{k+1}^{(r)}$. Given a truncation tolerance {\tt toltrank} and a maximum rank {\tt maxrank}, if $U\Sigma V^T$ is the SVD of $\bm{r}_l {\rm blkdiag}( \bm{\tau}_k, \bm{\alpha}_k) \bm{r}_r^T$ and $\{\sigma_j\}_{j=1,\ldots,d_k}$ are its singular values, we select
\begin{equation}\label{eq:rankselection}
    \iota_{k+1}:= \min \left\{ \mathtt{maxrank}, {\rm arg}\max_j\{\sigma_j\, : \,(\sigma_j/\sigma_1)\leq\mathtt{toltrank}\}\right\},
\end{equation}
and then set 
$X_{k+1}^{(l)}:=Q^{(l)} U_{\iota_{k+1}}$, $\bm{\tau}_{k+1}= \Sigma_{ \iota_{k+1}}$ and $X_{k+1}^{(r)}:=Q^{(r)} V_{ \iota_{k+1}}$ where the truncated SVD of $\bm{r}_l {\rm blkdiag}( \bm{\tau}_k, \bm{\alpha}_k) \bm{r}_r^T$ of rank $\iota_{k+1}$ is denoted $ U_{\iota_{k+1}} \Sigma_{ \iota_{k+1}}V_{\iota_{k+1}}^T$. 

The procedure for updating the matrices $\bmat{P}_{k}$ is analogous. These matrices are not stored in full format; their factors are immediately created and saved. However, additional care needs to be taken with the update and the truncation of the residual matrix $\bmat{R}_{k}$, especially when the number $p$ of terms in the matrix equation and/or the maximum value of {\tt maxrank} is large. This will be discussed in section~\ref{Randomized residual matrix}.

\vs{The scalars {\tt maxrank} and {\tt toltrank} need to be tuned to reach the best performance of the method. If {\tt maxrank} is chosen to be too small, then the method may stagnate without reaching the desired final accuracy, and similarly if the applied truncation is too severe (i.e, a too large value of {\tt toltrank} is chosen). The stagnation phenomenon is a well known possible shortcoming of low-rank iterations, and it has been reported and analyzed in the relevant literature, see, e.g., \cite{Kressner.Plesinger.Tobler.14},\cite{KressnerTobler2011},\cite{Simoncini.Hao.23}. The parameter {\tt maxrank} in fact affects not only the solution factors, but all iterates involved in the computation, including the direction matrices. Hence, to avoid stagnation, a sufficiently large value of  {\tt maxrank} should be allowed, according to the available memory capacity. Tuning {\tt toltrank} may be easier to do; our computational experience has shown that a value of the order of $10^{-10}-10^{-8}$ may in general suffice. A priori information, if available, on the rank properties of the continuous or discrete problem solutions may also help to tune these quantities.}

\subsection{Computing ${\pmb\alpha}_k$ and ${\pmb \beta}_k$}\label{sec:alphabeta}

The computation of both $\bmat{\alpha}_k$ and $\bmat{\beta}_k$ requires the solution of a multiterm matrix equation. Thanks to the explicit projection onto $\text{range}(P_k^{(r)}\otimes P_k^{(l)})$, the coefficient matrices of this equation have smaller dimensions than the original $\bmat{A}_{i}$ and $\bmat{B}_{i}$. On the other hand, the application of the operator $\mathcal{L}^*\mathcal{L}$ squares the number of terms leading to an equation with $p^2$ terms. In passing, we note that this cost occurs in computing the (same) coefficient matrix for both $\bmat{\alpha}_k$ and $\bmat{\beta}_k$, but also  the right-hand side for $\bmat{\beta}_k$. Therefore, the computation of $\bm{\alpha}_k$ and $\bm{\beta}_k$ must be handled with care to avoid excessive computational costs. Using the rank $q_k$ of $P_k^{(l)}$ (and $P_k^{(r)}$), we propose two strategies for this task. 

\vskip 0.1in 
\begin{itemize}
 \item[(i)]  If $q_k$ is sufficiently small, the explicit construction of the Kronecker form of~\eqref{eqn:projectedeq_alpha} is feasible and the resulting SPD linear system can be solved using Cholesky factorisation. The main cost of this procedure lies in the computation of the $p^2$ Kronecker products in the coefficient matrix,
\begin{equation}\label{eq:inner_matrix_alpha_beta}
    \mathfrak{T}=\sum_{i=1}^p \sum_{j=1}^p \bigg((P_k^{(r)})^T\bmat{B}_i\bmat{B}_j^TP_k^{(r)}\bigg)\otimes \bigg((P_k^{(l)})^T\bmat{A}_i^T\bmat{A}_jP_k^{(l)}\bigg).
\end{equation}
    Indeed, even for very sparse matrices $\bmat{A}_i$ and $\bmat{B}_i$, $\mathfrak{T}$ is in general dense. Moreover, computing the $p^2$ terms may already be problematic for moderate values of $p$. One advantage, however, is that the coefficient matrix (and its Cholesky factor) employed for the computation of $\bm{\alpha}_k$ can be reused to compute $\bm{\beta}_k$ if needed. 

\item[(ii)]  If $q_k$ becomes too large,  CG can be used\footnote{The matrix-vector operation is performed in matrix-matrix form, without explicitly forming the full  coefficient matrix $\mathfrak{T}$ in Kronecker form.} to solve \eqref{eqn:projectedeq_alpha}. 
While the $p^2$ terms $(P_k^{(l)})^T\bmat{A}_i^T\bmat{A}_jP_k^{(l)}$ and
    $(P_k^{(r)})^T\bmat{B}_i\bmat{B}_j^TP_k^{(r)}$ for $i,j=1,\ldots,p$, still need to be computed, this approach avoids assembling their Kronecker product. On the other hand, the normal equations nature of~\eqref{eqn:projectedeq_alpha} makes the latter prone to \sloppy{ill-conditioning}, slowing convergence of CG. Preconditioning is thus essential. More details about this key aspect will be given in section~\ref{sec:preconditioning}. Solving \eqref{eqn:projectedeq_alpha} iteratively results in an inexact computation of $\bm{\alpha}_k$ and $\boldsymbol{\beta}_k$, so that the orthogonality properties illustrated in section~\ref{sec:mr}--\ref{sscr} hold only approximately. The rank truncation steps in Algorithm~\ref{alg:subspaceCR} already affect the orthogonality properties, even when $\bm{\alpha}_k$ and $\boldsymbol{\beta}_k$ are computed exactly, so we expect the iterative solution of~\eqref{eqn:projectedeq_alpha} to
    \vs{have a somewhat limited extra negative effect.}

\end{itemize}

\subsection{Randomized truncation of the residual matrix}\label{Randomized residual matrix}
Writing $\bmat{X}_{k+1}$ as $X_{k+1}^{(l)} \bm{\tau}_{k+1} (X_{k+1}^{(r)})^T$,
the residual matrix $\bmat{R}_{k+1}=CD^T-\mathcal{L}(\bmat{X}_{k+1})$ can be written in factored form as
\begin{eqnarray}
\bmat{R}_{k+1}&=&[C,\bmat{A}_1X_{k+1}^{(l)},\ldots,\bmat{A}_pX_{k+1}^{(l)}]
\begin{bmatrix}
    I_q & & & \\
        & - \bm{\tau}_{k+1} & & \\
        &  & \ddots & \\
        & & & -\bm{\tau}_{k+1} \\
\end{bmatrix}
\begin{bmatrix}
D^T \\ (X_{k+1}^{(r)})^T\bmat{B}_1 \\ \vdots \\ (X_{k+1}^{(r)})^T\bmat{B}_p 
\end{bmatrix} \nonumber \\
&=& [C, \bm{A}_\star\bullet X_{k+1}^{(l)}]
\begin{bmatrix}
    I_q &  \\
        & - I_p\otimes \bm{\tau}_{k+1} \\
\end{bmatrix}
[D,\bmat{B}_\star^T \bullet X_{k+1}^{(r)}]^T. \label{eqn:Rbullet}
\end{eqnarray}
\noindent One could compute the skinny QR factorizations of the left and right factors followed by a truncated SVD of the resulting core matrix. However, the storage demand of fully allocating $[C, \bm{A}_\star\bullet X_{k+1}^{(l)}]$ and $[D,\bmat{B}_\star^T \bullet X_{k+1}^{(r)}]$ amounts to $2(p \cdot \mathtt{maxrank}+q)$ columns, assuming  the rank of $\bm{X}_{k+1}$ 
is {\tt maxrank}. If $p$ and/or {\tt maxrank} are sufficiently small so that the two matrices can be stored, this strategy is feasible and yields a low-rank representation of $\bmat{R}_{k+1}$. Otherwise, we adopt  a randomization strategy that allows us to compute both a low-rank approximation of $\bmat{R}_{k+1}$, and its Frobenius norm, as required by the stopping criterion in Algorithm~\ref{alg:subspaceCR}.  {A key tool for this is} randomized oblivious $(\varepsilon,\delta,m)$-subspace embeddings. 

{\color{black} A sketching matrix $S\in\mathbb{R}^{s\times n}$ (with random entries) is said to be a randomized oblivious $(\varepsilon,\delta,m)$-subspace embedding if, for any $m$-dimensional subspace $\mathcal{V}\subset\mathbb{R}^n$,
\begin{equation}\label{eq:def_sketching}
    (1-\varepsilon)\|v\|^2\leq\|Sv\|^2\leq (1+\varepsilon)\|v\|^2\quad \vs{\forall\, v\in\mathcal{V},}    
\end{equation}
holds 
with probability at least $1-\delta$; see, e.g.,~\cite[Section 2]{RandomGS} and~\cite[Section 8]{Martinsson_Tropp_2020}}. Depending on the  chosen  $S$, various values of the sketching dimension $s$, as a function of $\varepsilon$, $m$, and $\delta$, have been identified that ensure \eqref{eq:def_sketching} is satisfied; see, for example, \cite[Section 9]{Halko2010}. 
If {$S$ is chosen so that \eqref{eq:def_sketching} holds}, and {we replace the vector $v$ with any matrix $R$} such that range$(R)\subset \mathcal{V}$, then
\begin{equation}\label{eq:def_Fsketching}
(1-\varepsilon)\|R\|_F^2\leq\|S R\|_F^2\leq(1+\varepsilon)\|R\|_F^2 
\end{equation}
{also holds with probability at least $1-\delta$}.
Due to the low computational cost of applying them, in our numerical experiments in sections~\ref{sec:benchmark} and~\ref{sec:diffusion} we choose sketching matrices known as randomized subsampled trigonometric transformations (RSTTs); see e.g.,~\cite[Section 4.6]{Halko2010}. {\color{black}In particular, we consider
\begin{equation}\label{eq:our_sketching}
S=\sqrt{\frac{s}{n}}PCD,    
\end{equation}
where $D\in\mathbb{R}^{n\times n}$ is a diagonal matrix containing $\pm1$ at random on its main diagonal, $C\in\mathbb{R}^{n\times n}$ is the discrete cosine transform matrix, and $P\in\mathbb{R}^{s\times n}$ is a sampling matrix having as rows the transpose of $s$ randomly chosen canonical basis vectors of $\mathbb{R}^n$. A sequential application of the linear transformations in the right-hand side of~\eqref{eq:our_sketching} implies that the cost of performing $Sx$ amounts to $\mathcal{O}(n\log n)$ flops; this is the approach we follow. However, more sophisticated implementations can cut this cost down to $\mathcal{O}(n\log s)$ flops; see~\cite{WOOLFE2008335}.}
In~\cite{Tropp2011} theoretical guarantees have been obtained for RSTTs by selecting $s = \mathcal{O}(\varepsilon^{-2}(m + \log \frac{n}{\delta}) \log \frac{m}{\delta})$. However, {\color{black}in~\cite[Section 9]{Martinsson_Tropp_2020} it is stated that
 selecting the smaller sketching dimension $s = \mathcal{O}(\varepsilon^{-2} m/\delta)$ works
well in practice. }

{In our context we want to define 
sketching matrices $S_A \in \mathbb{R}^{s_A \times n_{A}}$ and $S_B \in \mathbb{R}^{s_B \times n_{B}}$ for $\text{range}([C, \bm{A}_\star\bullet X_{k+1}^{(l)}])$ and $\text{range}([D,\bmat{B}_\star^T \bullet X_{k+1}^{(r)}])$, respectively. Since we do not know the dimensions $m_A$ and $m_B$ of these spaces, we can employ the upper bound $m_A,m_B\leq (p\cdot\mathtt{maxrank}+q)$ to select {$s_A$ and $s_B$}. As this bound is likely to be pessimistic, it is reasonable to choose $s_A = s_B =s$} where $s=2(p \cdot \mathtt{maxrank}+q)$.

In the following, we identify a constant $\gamma$ that depends on the embedding parameters, such that $\|\bmat{R}_{k+1}\|_F \le \gamma \|S_A \bmat{R}_{k+1} S_B^T\|_F$ {holds with high probability}.  We can then use the singular value decomposition of the two-sided `sketched' matrix $S_A \bmat{R}_{k+1} S_B^T \in \mathbb{R}^{s_{A} \times s_{B}}$ 
to construct $\bmat{R}_{k+1}$ in low rank factored form and to estimate $\|\bmat{R}_{k+1}\|_F$. 
More precisely, 
we perform the
following steps:
\vskip 0.15in
\begin{itemize}

\item[1.] Compute the {skinny} QR decompositions of the sketched left and right factors in \eqref{eqn:Rbullet} to give
$$
S_A [C, \bm{A}_\star\bullet X_{k+1}^{(l)}] = Q_A R_A \quad \mbox{and}\quad 
S_B [\cep{D}, \bm{B}_\star^T\bullet X_{k+1}^{(r)}]=Q_B R_B,
$$
discard the $Q$'s, and implicitly write
{\small 
\begin{eqnarray}
\bm{R}_{k+1} & = &
\left([C, \bm{A}_\star\bullet X_{k+1}^{(l)}] R_A^{-1} \right )
\,
\left( R_A {\rm blkdiag}(I, - I\otimes \bmat{\tau}_{k+1}) R_B^T\right ) 
\,
\left([\cep{D}, \bm{B}_\star^T\bullet X_{k+1}^{(r)}] R_B^{-1} \right )^T \nonumber \\
&=:& K_l \bm{\rho} K_r^T. \label{eqn:KrK}
\end{eqnarray}
}

\noindent The pseudo-inverses of $R_A$ and $R_B$ may be employed if these matrices are singular or severely ill-conditioned.

\item[2.] Compute the SVD $\bm{\rho}=U \Sigma V^T$, select $\iota$ as in~\eqref{eq:rankselection},  
and {truncate {the former} to obtain $U_{\iota},\Sigma_{\iota},V_{\iota}$.}

\item[3.] Update the (truncated) residual in factored form as 
$\bm{R}_{k+1} = R_{k+1}^{(l)} \bm{\rho}_{k+1} (R_{k+1}^{(r)})^T$ where
$$
R_{k+1}^{(l)} : = K_l U_{\iota}, \quad 
\bm{\rho}_{k+1} := \Sigma_{\iota}, \quad
R_{k+1}^{(r)}:= K_r V_{\iota} .
$$
\end{itemize} 
\vskip 0.15in 
{\color{black} In Table~\ref{tab:costsketching} we summarize the main computational costs of the procedure outlined above for the case where $S_A$ and $S_B$ are RSTTs as in~\eqref{eq:our_sketching}.
\begin{table}[t!]
    \centering
    {\color{black}
    \begin{tabular}{r|l}
        Compute $S_A [C, \bm{A}_\star\bullet X_{k+1}^{(l)}]$   &  \multirow{ 2}{*}{$\mathcal{O}((p\cdot\mathtt{maxrank}+q)(n_A\log n_A+n_B\log n_B))$} \\
        and $S_B [D, \bm{B}_\star^T\bullet X_{k+1}^{(r)}]$& \\
        QR factorizations & $\mathcal{O}(2(p\cdot\mathtt{maxrank}+q)^3)$\\
        SVD & $\mathcal{O}((p\cdot\mathtt{maxrank}+q)^3)$\\
        Retrieve $R_{k+1}^{(l)}$ and $R_{k+1}^{(r)}$& $\mathcal{O}(\iota(n_A+n_B)(p\cdot\mathtt{maxrank}+q))$ 
    \end{tabular}}
    \caption{{\color{black} Costs of the randomized low-rank truncation scheme for the residual when $S_A\in\mathbb{R}^{s\times n_A}$ and $S_B\in\mathbb{R}^{s\times n_B}$ are both RSTTs as in~\eqref{eq:our_sketching} with $s=2(p\cdot\mathtt{maxrank}+q)$.}}
    \label{tab:costsketching}
\end{table}
}

{Let (\ref{eq:def_Fsketching})} hold for $S_A$ applied to range$(K_l) \subseteq {\rm range}([C, \bm{A}_\star\bullet X_{k+1}^{(l)}])$, and for $S_B$ applied to range$(K_r) \subseteq {\rm range}([\cep{D}, \bm{B}_\star^T\bullet X_{k+1}^{(r)}])$ with embedding parameters $\varepsilon_A$, $\delta_A$ and $\varepsilon_B$, $\delta_B$, respectively. Using the form of the exact residual matrix from (\ref{eqn:KrK}) gives
\begin{eqnarray}
\|S_A \bmat{R}_{k+1} S_B^T \|_F^2 &=&
\|S_A (K_l \bmat{\rho} K_r^T S_B^T) \|_F^2 \ge (1-\varepsilon_A)
\|K_l \bmat{\rho} K_r^T S_B^T \|_F^2 \label{eqn:twosidedbound}\\
&\ge & 
(1-\varepsilon_A) (1-\varepsilon_B) \|K_l \bmat{\rho} K_r^T  \|_F^2 =:
\frac 1 {\gamma^2} \| \bmat{R}_{k+1} \|_F^2 \nonumber
\end{eqnarray}
 with $\gamma = (1-\varepsilon_A)^{-1/2} (1-\varepsilon_B)^{-1/2}$ {(see also \cite[Theorem 3.11]{MeierPhD.24} for a similar result)}, and the final inequality  holds with probability at least
 $(1-\delta_A)(1-\delta_B)$.
Hence, we have {the probabilistic upper bound}
\begin{equation*}
\| \bmat{R}_{k+1} \|_F \le \gamma \|S_A \bmat{R}_{k+1} S_B^T \|_F { = \gamma \, \|\Sigma\|_F}
\end{equation*}
with $\Sigma$ computed in step 2. {We refer to \cite{Meier.Nakatsukasa.24} for a more general analysis of truncation quality using left and right sketchings.}

If the $\bmat{B}_j$'s are small (i.e., if $n_{B}$ is small), there is no need to introduce the sketching matrix $S_B$ to reduce the dimension of $K_r$. In this
case, we replace (\ref{eqn:KrK}) with $\bm{R}_{k+1} = K_l \bm{\rho}$, where
$\bm{\rho}={\rm blkdiag}(I, - I\otimes \bmat{\tau}_{k+1})  \,
\left([\cep{D}, \bm{B}_\star^T\bullet X^{(r)}] \right )^T$,  and proceed with step 2. In step 3, we
simply define $R_{k+1}^{(r)} := V_{\iota}$.  Proceeding as in (\ref{eqn:twosidedbound}), the bound
$$
\|S_A \bmat{R}_{k+1} \|_F^2 
\ge (1-\varepsilon_A)
\|K_l \bmat{\rho} \|_F^2
$$ 
holds with probability at least $(1-\delta_A)$, from which 
$\| \bmat{R}_{k+1} \|_F \le \gamma \, \|\Sigma\|_F$ holds with the same probability,
 with $\gamma=(1-\varepsilon_A)^{-1/2}$. 
 An analogous procedure can be applied if the $\bmat{A}_j$'s (but not the $\bmat{B}_j$'s) are small.

{\color{black} \begin{remark}
{\rm
In all the experiments reported in sections~\ref{sec:benchmark}--\ref{sec:diffusion} we fix the sketching operators $S_A$ and $S_B$ once and for all at the beginning of the iterative solution process. 
This is justified by recalling that, by construction, the spaces we need to embed, ${\rm range}([C, \bm{A}_\star\bullet X_{k+1}^{(l)}])$ and ${\rm range}([D, \bm{B}_\star^T\bullet X_{k+1}^{(r)}])$, have dimension $p\cdot\mathtt{maxrank}+ q$ at most for any $k$
and that our $S_A$ and $S_B$ embed any 
$(p\cdot\mathtt{maxrank}+ q)$-dimensional subspace with high probability.
An alternative approach would be to draw new 
sketching operators $S_A$ and $S_B$ at each iteration $k$. This second strategy may imply smaller failure probabilities. However, in our numerical experiments fixing $S_A$ and $S_B$ once and for all leads to a highly robust solution process. 
}
\end{remark}
}

\subsection{Memory requirements}\label{Memory requirements}

In this section we summarize the memory requirements of \sscr\ and \ssmr. 

Recall that $\bX_k$, $\bR_k$, and $\bP_k$ are kept in low-rank factored form $V^{(l)} \bmat{\nu}(V^{(r)})^T$, with $V^{(l)}, V^{(r)}$ having at most {\tt maxrank} columns. In \sscr, allocating these iterates therefore requires storing up to $3((n_A+n_B)\cdot\mathtt{maxrank}+\mathtt{maxrank}^2)$ entries. This reduces to $2((n_A+n_B)\cdot\mathtt{maxrank}+\mathtt{maxrank}^2)$ entries for \ssmr, since $\bP_k=\bR_k$. Two matrices of size 
$n_A\times \mathtt{maxrank}$ and $n_B\times \mathtt{maxrank}$ are also required for working storage.

If $\bm{\alpha}_k$ and $\bm{\beta}_k$ are computed with a direct solver
(case (i) in section~\ref{sec:alphabeta}), then the
 full matrix $\mathfrak{T}$
 of size $\mathtt{maxrank}^2\times \mathtt{maxrank}^2$  defined in \eqref{eq:inner_matrix_alpha_beta} needs to be stored.
If $q_k$ is large (case (ii) in section~\ref{sec:alphabeta}) the main storage allocation demand is that of the $p^2$ dense matrices $(P_k^{(l)})^T\bmat{A}_i^T\bmat{A}_jP_k^{(l)}$ and $(P_k^{(r)})^T\bmat{B}_i\bmat{B}_j^TP_k^{(r)}$ for $i,j=1,\ldots,p$, of size at most $\mathtt{maxrank}\times \mathtt{maxrank}$ each. 
Employing the PCG method for the inner solves to compute $\bm{\alpha}_k$ and $\bm{\beta}_k$ inexactly does not significantly increase this demand.

Finally, as discussed in section~\ref{Randomized residual matrix}, the computation of the low-rank factorization of the residual matrix $\bR_k$ is performed via sketching when at least one of $n_A$ and $n_B$ is large. The main storage cost {is for the left and right factors which have a total of }
$(p\cdot\mathtt{maxrank}+q)(s_A+n_B)$ entries for one-sided sketching when, say,
$n_B<s_B$ and $(p\cdot\mathtt{maxrank}+q)(s_A+s_B)$ entries when two-sided sketching is applied.

\section{Preconditioning}\label{sec:preconditioning}

As in the vector case, acceleration procedures can be  employed for multiterm matrix equation solvers, see, e.g.,  \cite{Kressner.Plesinger.Tobler.14}, \cite{Biolietal2024}, \cite{Palittaetal2025}. Following similar derivations in \cite{Biolietal2024} and \cite{Palittaetal2025}, natural preconditioners are obtained when a `leading' part ${\cal P}$ of the operator ${\mathcal{L}}$ can be identified. That is, by splitting the operator as ${\cal L} = {\cal P} - {\cal N}$, one can use the leading term, or a cheaper approximation $\widetilde {\cal P}$ thereof, as a preconditioner. At each iteration, the action of the inverse is then applied as
$\widetilde {\cal P}^{-1}(\bmat{R})$ where $\bmat{R}$ is always stored in low-rank format.

Before describing how to efficiently incorporate preconditioning into our new algorithms, we need to discuss what types of preconditioners are feasible for matrix equations \eqref{eq:main} with many terms $(p>2)$. There are two main considerations. The first is that the identified splitting should lead to an effective preconditioner; this concern is typical of fixed-point type iterations. The second is that applying the action of ${\cal P}^{-1}$ (or an approximation thereof) should incur an acceptable computational cost. This will depend on the number of addends in the designated leading part of the operator. In 
the recent literature, two settings have been considered:
\vskip 0.1in
\begin{itemize}
\item[(i)] {\it One-term preconditioning}. In this case, a pair $(\bmat{A}_i, \bmat{B}_i)$ is identified as a preconditioner, so that ${\cal P}^{-1}(\bmat{R}) = \bmat{A}_i^{-1}\bmat{R}\bmat{B}_i^{-1}$. In the event that computing the action of the inverse of $\bmat{A}_i$ or $\bmat{B}_i$ is too expensive, these matrices may be replaced with suitable approximations.
See section~\ref{sec:diffusion} for an illustration of one-term preconditioning in a concrete application setting.

\item[(ii)] {\it Two-term preconditioning}. Here, we assume that there exist two pairs of coefficient matrices $(\bmat{A}_i, \bmat{B}_i)$ and $(\bmat{A}_j, \bmat{B}_j)$ such that
${\cal P}(\bmat{R})=\bmat{A}_i \bmat{R} \bmat{B}_i + \bmat{A}_j \bmat{R} \bmat{B}_j$ denotes the leading part of ${\mathcal L}$ applied to $\bmat{R}$. Inverting ${\cal P}$ corresponds to solving a Sylvester equation at each iteration, which is very expensive. In this case, following  \cite{Biolietal2024} and \cite{Palittaetal2025} we replace ${\cal P}^{-1}$ with the operator $\widetilde{\cal P}^{-1}$ that corresponds to applying a fixed number of ADI iterations\footnote{{ADI is an iterative method based on rational Krylov subspaces \cite{ADI_Sylv2009},\cite{Ellner1991}}.}.
See section~\ref{sec:benchmark} for an illustration of two-term preconditioning. Other solvers for Sylvester equations may be considered in place of ADI, depending on the properties of  $\cal P$.
\end{itemize}

\vskip 0.1in

A major problem arises if the leading term contains more than two addends, as applying the action of ${\cal P}^{-1}$ may then be as difficult as solving the original problem. In this case, an inner-outer procedure seems to be feasible. Here, one identifies $\widetilde{\cal P}^{-1}$ with a few iterations of another iterative method, or the same one, applied to the designated leading operator, or to the whole operator \cite[section 9.4.1]{Saad2003}.
An interesting preconditioning alternative was recently proposed in \cite{Voet.25},
where the possibility of an approximate inverse with a Kronecker structure is explored.

After line 8 of Algorithm~\ref{alg:subspaceCR} the action of the chosen preconditioner can be included in \sscr\ as follows
$$
\bmat{Z}_{k+1} =\widetilde{\cal P}^{-1}(\bmat{R}_{k+1}).
$$
Then, line 10 should be replaced by
$$
10. \hskip 0.5in \bmat{P}_{k+1} = \bmat{Z}_{k+1}  + P_k^{(l)} \bmat{\beta}_k (P_k^{(r)})^T
$$
where $\bmat{\beta}_k$ is computed accordingly. In \ssmr\, we 
simply need to set
$\bmat{P}_{k+1} = \bmat{Z}_{k+1}$. For both algorithms, we set  
$\bmat{P}_{0} = \bmat{Z}_{0}$.
In other words, the sequence $\{\bmat{P}_{k} \}_{k\ge 0}$ generates a recurrence of preconditioned spaces.

Regarding the memory requirements of preconditioning, since $\bZ_k$ is also kept in low-rank factored form, it has at most 
$(n_A+n_B)\cdot\mathtt{maxrank}+\mathtt{maxrank}^2$ entries to be stored. One-term preconditioning does not require additional memory; see (i) above.
If two-term preconditioning is used, more memory may need to be allocated. For instance, if
$\widetilde{\mathcal{P}}^{-1}$ is applied by performing $t_{ADI}$ iterations of the ADI method, {then the memory allocation demand increases by $t_{ADI}(n_A+n_B)\cdot\mathtt{maxrank}$;} { see, e.g.,~\cite[Section 2]{BENNER_SylvADI2014} for more details on the ADI memory cost.} 

{
\section{\vs{Theoretical considerations}}\label{sec:analysis}
In this section, we \cep{investigate} the \vs{theoretical} properties of the considered recurrences. Using the residual minimization properties and the \ssmr\ recurrence, \cep{we can relate the norms of two successive residual matrices.}
In the following, $\cal A$ is as defined in (\ref{eqn:Kron}).

\begin{proposition}\label{Prop6point1}
Assume that $\mathcal A$ is positive definite and 
let $\bmat{R}_k=R_k^{(l)} (R_k^{(r)})^T$ with $\bmat{R}_k\ne 0$. 
Let $\mu({\cal A}) = \lambda_{\min}( \frac 1 2 ({\cal A}+{\cal A}^T))$ \vs{ and ${\cal M}_{Q_k}= Q_k^T{\cal A}^T {\cal A}Q_k$, where the columns of $Q_k$ span the range of $R_k^{(r)}\otimes R_k^{(l)}$.}
Then after one \ssmr\  iteration it holds
\begin{equation}\label{eqn:elman_bound}
\|\bmat{R}_{k+1}\|_F  \le
	\left( 1 -  \frac{\mu({\cal A})^2}{\vs{\|{\cal M}_{Q_k}\|}}\right)^{\frac 1 2}
\|\bmat{R}_{k}\|_F .
\end{equation}
\vs{Since $\|{\cal A}\|^2 \ge  \|{\cal M}_{Q_{k}}\|$, it also follows
that $\|\bmat{R}_{k+1}\|_F  \le
	\left( 1 -  \frac{\mu({\cal A})^2}{\|{\cal A}\|^2}\right)^{\frac 1 2}
\|\bmat{R}_{k}\|_F$.}
\end{proposition}

\begin{proof}
From $\bmat{R}_{k+1} = \bmat{R}_k - 
\mathcal{L}(R_k^{(l)} \bm{\alpha}_k (R_k^{(r)})^T)$ we obtain
\begin{eqnarray*}
\|\bmat{R}_{k+1}\|_F^2 &= &
{\rm trace}(\bmat{R}_{k+1}^T \bmat{R}_{k}) -
{\rm trace}(\bmat{R}_{k+1}^T \mathcal{L}(R_k^{(l)} \bm{\alpha}_k (R_k^{(r)})^T)) \\
&=&
{\rm trace}(\bmat{R}_{k+1}^T \bmat{R}_{k}) =
{\rm trace}(\bmat{R}_{k}^T \bmat{R}_{k}) -
{\rm trace}(\bmat{R}_{k}^T \mathcal{L}(R_k^{(l)} \bm{\alpha}_k (R_k^{(r)})^T)) \\
&=&
\| \bmat{R}_{k}\|_F^2\left ( 1 - 
\frac{{\rm trace}(\bmat{R}_{k}^T \mathcal{L}(R_k^{(l)} \bm{\alpha}_k (R_k^{(r)})^T))%
}{ \|\bmat{R}_{k}\|_F^2 }\right ) .
    \end{eqnarray*}
In the second equality, we have used the orthogonality property
of $\bmat{R}_{k+1}$.
We next bound from below the second term in parentheses.
Let $W=R_k^{(r)}\otimes R_k^{(l)}$ and $e={\rm vec}(I_q)$, so that
$r={\rm vec}(\bmat{R}_{k})=W e$.
Then
\begin{equation}\label{eqn:ratio}
\frac{{\rm trace}(\bmat{R}_{k}^T \mathcal{L}(R_k^{(l)} \bm{\alpha}_k (R_k^{(r)})^T))%
}{ \|\bmat{R}_{k}\|_F^2 } =
\frac{r^T {\cal A} W \underline{\bmat{\alpha}} }{r^Tr}
\end{equation}
where $\underline{\bmat{\alpha}} = \textrm{vec}(\bm{\alpha}_k)$. Writing the equation for $\bm\alpha_{k}$ in vector form as
$W^T{\cal A}^T {\cal A}W \underline{\bmat{\alpha}}=W^T {\cal A}^T r$, 
and setting ${\cal M}= W^T{\cal A}^T {\cal A}W$ then gives $\underline{\bmat{\alpha}} = {\cal M}^{-1}(W^T {\cal A}^T r)$. Substituting this expression into the right-hand side of (\ref{eqn:ratio}) we obtain

$$
\frac{{\rm trace}(\bmat{R}_{k}^T \mathcal{L}(R_k^{(l)} \bm{\alpha}_k (R_k^{(r)})^T))%
}{ \|\bmat{R}_{k}\|_F^2 } =
\frac{r^T {\cal A} W  {\cal M}^{-1}W^T {\cal A}^T r }{r^T r} .
$$
\vs{Let} $W=\vs{Q_k}\bmat{\rho}$ be the reduced QR decomposition of $W$\vs{, and}
${\cal M}_{\vs{Q_k}}:=\vs{Q_k}^T{\cal A}^T{\cal A}\vs{Q_k}$%
\vs{; In the following we omit the subscript $k$.} Then
$$
\frac{r^T {\cal A} W  {\cal M}^{-1}W^T {\cal A}^T r }{r^T r}
=
\frac{r^T {\cal A} Q  {\cal M}_Q^{-1}Q^T {\cal A}^T r }{r^T r}
= \frac{r^T {\cal A} Q  {\cal M}_Q^{-1}Q^T {\cal A}^T r }{r^T {\cal A} Q Q^T {\cal A}^T r} \, \,  \frac{r^T {\cal A} Q  Q^T {\cal A}^T r }{r^Tr}. 
$$

For the numerator of the second factor, we let $r=We=Q\bmat{\rho}e=:Q s$, so
that $\|r\|=\|s\|$. We observe that 
$|r^T {\cal A}^T r |=
|s^T Q^T {\cal A}^T r | \le \|s\| \| Q^T {\cal A}^T r\|$ and in particular, that
$\|Q^T {\cal A}^T r\|\ne 0$ for
$r\ne 0$. Hence, 
\begin{eqnarray*}
r^T {\cal A} Q  Q^T {\cal A}^T r 
&= & 
    \|Q^T {\cal A}^T r\|^2 
    \ge \frac{\left|r^T {\cal A}^T r \right|^2}{\|r\|^2}  
	 = \frac{\left|\frac 1 2 r^T ({\cal A}^T + {\cal A}) r \right|^2}{\|r\|^2}  \ge \mu({\cal A})^2 \|r\|^4 / \|r\|^2 .
\end{eqnarray*}
The first factor is bounded as
\begin{equation}\label{eqn:normA}
\frac{r^T {\cal A} Q  {\cal M}_Q^{-1}Q^T {\cal A}^T r }{r^T {\cal A} Q Q^T {\cal A}^T r} \ge
	\frac 1 {\|{\cal M}_Q\|}\vs{.}
\end{equation}
In summary, we have obtained the following bound
$$
\frac{{\rm trace}(\bmat{R}_{k}^T \mathcal{L}(R_k^{(l)} \bm{\alpha}_k (R_k^{(r)})^T))%
}{ \|\bmat{R}_{k}\|_F^2 } \ge 
\vs{\frac 1 {\|{\cal M}_Q\|}}
\frac{\mu({\cal A})^2 \|r\|^4}{\|r\|^4}.
$$
Therefore,
$$
\| \bmat{R}_{k+1}\|_F^2 =
\| \bmat{R}_{k}\|_F^2\left ( 1 - 
\frac{{\rm trace}(\bmat{R}_{k}^T \mathcal{L}(R_k^{(l)} \bm{\alpha}_k (R_k^{(r)})^T))%
}{ \|\bmat{R}_{k}\|_F^2 }\right ) \le
\| \bmat{R}_{k}\|_F^2
\vs{\left ( 1 - \frac{\mu({\cal A})^2}{ \|{\cal M}_Q\|}\right )}.
$$
\vs{Since $\|{\cal A}\|^2 \ge  \|{\cal M}_{Q_{k}}\|$, the second result follows.}
\end{proof}
}

\vs{We remark that using arguments similar to those for the vector setting \cite[page 34]{Greenbaum1997}, the residual norm for \sscr\ is also reduced, by at least as much as for \ssmr.} 
\vs{The results of Proposition~\ref{Prop6point1} imply}
that if the residual matrix $\bm{R}_{k}$ is not truncated, 
\vs{$\|\bmat{R}_{k}\|_F$ converges to zero as $k\to \infty$,}
since $\mu({\cal A}) < \|{\cal A}\|$
and
\begin{equation*}\label{eqn:elman_bound2}
\|\bmat{R}_{k}\|_F  \le
	\left( 1 -  \frac{\mu({\cal A})^2}{\|{\cal A}\|^2}\right)^{\frac k 2}
\|\bmat{R}_{0}\|_F .
\end{equation*}
This \vs{analysis} is in line with known classical results for Generalized Conjugate Residual type methods \cite{Vatsya.88}.
\vs{The presence of $\|{\cal M}_Q\|$ in place of $\|{\cal A}\|^2$ in (\ref{eqn:elman_bound}) suggests that faster decay can occur. In particular, the bound in (\ref{eqn:elman_bound})  takes into account the presence of the subspace projection step associated with having a {\it matrix} $\bmat{\alpha}_k$.}

Unfortunately, in general, the \vs{bound in (\ref{eqn:elman_bound})}  may not be descriptive of the actual convergence of the method, \vs{as is the case for the classical  result.}

\vskip .1in
We now discuss the behavior of the matrix iterations in terms of the computed subspaces, with a view to interpreting the methods once truncation is incorporated. 
Recall the notation 
$\bA_\star \bullet R =[\bA_1R, \ldots, \bA_p R]$,  
and analogously for $\bB_\star^T\bullet  R$ with $R$ of conforming dimensions. Moreover, note that for $k\ge 0$, subsequent applications of the operator can be written as
$\bA_\star^{k+1} \bullet R=
\bA_\star \bullet (\bA_\star^k \bullet R)$. 
We then define the space 
$\Kr_k(\bA_\star,R_0)= {\rm range}([R_0, \bA_\star \bullet R_0, \ldots, \bA_\star^k \bullet R_0])$ (as in \cite{Palittaetal2025}).
Note that the spaces are nested, that is
$\Kr_k(\bA_\star,R_0)\subseteq \Kr_{k+1}(\bA_\star,R_0)$.

After $k$ iterations of either \sscr\ or \ssmr\ and without forced rank truncation to \texttt{maxrank}, the columns of 
 ${R}_{k+1}^{(l)}$,  ${P}_{k+1}^{(l)}$ span $\Kr_k(\bA_\star,R_0^{(l)})$, and similarly, the columns of
${R}_{k+1}^{(r)}$, ${P}_{k+1}^{(r)}$ span $\Kr_k(\bB_\star^T,R_0^{(r)})$. The dimension of these spaces quickly grows due to the inclusion of many terms\footnote{The actual dimension growth at each iteration depends both on $p$ and {\tt maxrank}, but also on the linear independence of the added columns with respect to the already computed space.} as $k$ increases,
 although it may grow less than one would expect, due to possible redundancies.
 Recalling the derivation of the recurrence coefficient $\bmat{\alpha}_k$ (e.g., Proposition~\ref{prop:alpha} for \ssmr), it follows that before truncation is enforced, both methods \ssmr\ and \sscr\ perform a matrix Petrov-Galerkin projection onto the spaces $\bA_\star\bullet \Kr_k(\bA_\star,R_0^{(l)})$ (from the left) and  $\bB_\star^T\bullet \Kr_k(\bB_\star^T,R_0^{(r)})$ (from the right)\footnote{This consideration is well known in the vector case, and it corresponds to the mathematical equivalence of GCR methods with GMRES.}. 
 For $k$ sufficiently large so the maximum allowed rank of $\Kr_k(\bA_\star,R_0^{(l)})$ and 
 $\Kr_k(\bB_\star^T,R_0^{(r)})$ is reached, truncation is enforced, yielding the reduced subspaces 
 $
 {\rm range}(R_{k+1}^{(l)})
 \subset \Kr_{k+1}(\bA_\star,R_0^{(l)})$,
 $ {\rm range}( R_{k+1}^{(r)})
 \subset\Kr_{k+1}(\bB_\star,R_0^{(r)})$ of dimension {\tt maxrank}. The Petrov-Galerkin projection onto these subspaces will continue to decrease the residual norm as long as new information is injected into the subspaces after truncation, compared with the previous iterate. 
 Before the first forced truncation takes place, this condition can be formally written as
 $$
 {\rm range}( R_{k+1}^{(l)})
 \not\subset
 \Kr_{k}(\bA_\star,R_0^{(l)});
 $$
 and similarly for 
  ${\rm range}(R_{k+1}^{(r)})$.
After the first rank truncation to \texttt{maxrank},
 the condition above can be rewritten as
${\rm range}( R_{k+1}^{(l)})
 \not\subset
 {\rm range}( R_{k}^{(l)})$.
 
While the recurrence before truncation corresponds to a projection method with a growing subspace, 
the process after truncation may be interpreted as a {\it thick} restarting procedure, which is commonly used in projection methods for large eigenvalue problems and linear systems; see\footnote{Depending on the strategy adopted to retain vectors, the term ``Implicitly restarted methods``  is often employed.}, e.g., \cite[section 9.3]{Watkins.07}, \cite{doi:10.1137/S0895479897321362}. 
This procedure acts as follows: After a fixed number of iterations, the projection phase is stopped, then the current approximation space is reduced to a significantly smaller dimension - ensuring that relevant information is retained,  and finally the process is restarted by adding new vectors to this retained {\it thick} vector (in fact a tall matrix).
Each restart is called a cycle. In our setting, after the truncation that yields $R_{k+1}^{(l)}$, the new columns $\bmat{A}_\star\bullet R_{k+1}^{(l)}$ are added to the subspace, giving range$([R_{k+1}^{(l)}, \bmat{A}_\star\bullet R_{k+1}^{(l)}])$; this new subspace most likely again requires truncation. As the subsequent iterations proceed, the space dimension keeps changing in an accordion-like manner. In summary, once the truncation process is installed, every new iteration behaves like a cycle of thick restarting, and each restart consists of a single iteration.

\section{Numerical experiments on a benchmark problem}\label{sec:benchmark}
In this section, we consider a benchmark problem that is commonly used in the literature to test methods for solving (\ref{eq:main}) when the associated Kronecker matrix ${\cal A}$ is nonsymmetric and nonsingular. The aim here is to describe the general behavior of the two new methods \ssmr\ and \sscr, and compare their performance with that of  state-of-the-art algorithms. All experiments were performed in MATLAB on a modest MacBook Pro laptop with a 2.6GHz 6-Core Intel Core i7 processor and 16GB RAM. 

In all experiments we fix $\bX_0={0}$ and $\mathtt{tol}=10^{-6}$. As discussed in section \ref{Randomized residual matrix}, {when one or both of the dimensions $n_A$ or $n_{B}$ is too large}, we use a randomization strategy to compute the norm of the residual in the stopping condition.
 We set $s=2 (p\cdot \mathtt{maxrank}+q)$ and
\vskip 0.05in
\begin{itemize}
    \item[(i)] if {$n_A,n_B<s$}, we compute $\|\bR_{k+1}\|_F$;
    \item[(ii)] {if $n_A \ge s$ but $n_B < s$}, we compute $\|S_A\bR_{k+1}\|_F$;
    \item[(iii)] if {$n_A,n_B \ge s$} we compute $\|S_A\bR_{k+1}S_B^T\|_F$
\end{itemize}
\vskip 0.05in
where we use RSTTs (see section~\ref{Randomized residual matrix})
as sketching matrices $S_A\in\mathbb{R}^{s\times n_A}$ and $S_B\in\mathbb{R}^{s\times n_B}$ (if required) and $s_A=s_B=s$. We also report the true relative residual Res:=$\| \bmat{R}_{k}\|_{F}/ \| \bmat{R}_{0}\|_{F}$ for the obtained solution, to verify the reliability of our randomization strategy. 

Our implementation of Algorithm~\ref{alg:subspaceCR}  automatically switches between an exact and inexact computation of $\bm{\alpha}_k$ and $\bm{\beta}_k$: if $q_k$, the rank of the current $P_k^{(l)}$ and $P_k^{(r)}$, is such that $q_k^2<4000$, then we assemble the matrix $\mathfrak{T}$ in~\eqref{eq:inner_matrix_alpha_beta} and solve the related SPD linear system by computing its Cholesky factorization. Otherwise, we use preconditioned CG ({\sc pcg}) on the Kronecker formulation of (\ref{eqn:doublesum}), with relative residual norm tolerance $10^{-4}$. In our benchmark tests, we choose the following two-term preconditioning operator for this inner {\sc pcg} solve
{\small 
$$
{\cal P}(\bmat{\alpha})= (R_k^{(l)})^T\bmat{A}_1^T\bmat{A}_1 R_k^{(l)}\bm{\alpha}(R_k^{(r)})^T\bmat{B}_1\bmat{B}_1^T R_k^{(r)} + (R_k^{(l)})^T\bmat{A}_2^T\bmat{A}_2 R_k^{(l)}\bm{\alpha}(R_k^{(r)})^T\bmat{B}_2\bmat{B}_2^T R_k^{(r)}
$$ 
}
which is the leading operator of the projected equation for the considered problem.

\subsection{A convection-diffusion problem}
We  consider the following steady state convection-diffusion boundary value problem $$
\left\{\begin{array}{l}
          -\cep{\nu}\Delta u+\vec{w}\cdot \nabla u=f,\quad\text{in }D=(-1,1)^2,\\
          u(1,y)=u(x,-1)=u(x,1)=0,\\
          u(-1,y)=1\\
         \end{array}\right .
$$
\noindent with constant source term $f=1$ and recirculating wind field,
$$
\vec{w}(x,y)=(\phi_1(x)\psi_1(y),\phi_2(x)\psi_2(y))=(2y(1 - x^2), -2x(1 - y^2)).
$$
Following \cite{Palitta.Simoncini.16}, we apply standard centered finite differences for the first and second derivatives on a uniform mesh of points $(x_i, y_j)_{i,j=0,\ldots,n-1}$ with spacing $h=2/(n-1)$ in each direction. Denoting with $\bmat{X} \in \mathbb{R}^{n \times n}$ the matrix whose entries approximate $u(x_i,y_j)$, this discretization leads to the matrix equation 
\begin{equation}\label{eq:convdiff}
\bmat{A}_1 \bmat{X}+\bmat{X}\bmat{B}_2+ (\bmat{\Phi}_1 \bmat{E}_1) \bmat{X}\bmat{\Psi}_1+\bmat{\Phi}_2\bmat{X} (\bmat{E}_2\bmat{\Psi}_2) =C D^T,
\end{equation}
where $\bmat{A}_1$ and $\bmat{B}_2$ correspond to the discretized second derivatives in the $x$ and $y$ directions, respectively, while the other terms are related to the first derivatives. We test our solvers on \eqref{eq:convdiff} for different values of the  mesh parameter $h$ {(equivalently, $n$)} and the diffusion coefficient
$\cep{\nu}$. We apply the two-term preconditioner
\begin{equation}\label{eq:precond}
\mathcal{P}:\bmat{X}\rightarrow \bmat{A}_1\bmat{X}+\bmat{X}\bmat{B}_2,
\end{equation}
corresponding to the (discrete) diffusion part of the operator and approximate the action of $\mathcal{P}^{-1}$ via 8 iterations of the low-rank ADI method \cite{BENNER_SylvADI2014} with (sub)optimal Wachspress’ shifts~\cite{ADIshifts}. 

In Table~\ref{Tab-ConvDiff} we compare the performance of \sscr\ and \ssmr\  with that of low-rank GMRES ({\sc{lr-pgmres}})~\cite{PalittaKuerschner2021} for $\cep{\nu}=0.1, 0.01$ and for various values of the problem dimension $n=n_A=n_B$. {We apply the same two-term preconditioner for all methods}.
In each case, we report the number of iterations $k$, the CPU time in seconds, 
 and the true relative residual norm (Res) at termination. For \sscr\ and \ssmr, we adapt the value of {\tt maxrank} to the choice of $\cep{\nu}$. 
Specifically, we set $\mathtt{maxrank}=50$ for $\cep{\nu}=0.1$ and $\mathtt{maxrank}=70$ for $\cep{\nu}=0.01$. 
Due to the higher chosen value of $\mathtt{maxrank}$, we use {\sc pcg} to compute $\bm{\alpha}_k$ (and $\bm{\beta}_k$) when $\cep{\nu}=0.01$. In this case, we also report the minimum and maximum number of iterations (in square brackets) required for this inner solve (column `{\sc pcg}'). For $\cep{\nu}=0.1$, we construct~\eqref{eq:inner_matrix_alpha_beta} and solve the related SPD linear system using Cholesky factorization; {\sc pcg} is not employed. The rank of the final iterate  $\bX_k$ is reported in the column `Rank'. {For {\sc{lr-pgmres}},
 the low-rank factors representing the basis of the constructed subspace need to be stored. In the column `Mem' in Table~\ref{Tab-ConvDiff}, the number of stored $n$-dimensional vectors is reported, resulting in  $\mathcal{O}(\text{Mem} \cdot n)$ of memory allocations.}

\tabcolsep.95pt 
\begin{table}[tb]
\begin{center}
\begin{tabular} {| r  | c  cc c   c | cc  c c   c | c  c c r |  } \hline 
      \multicolumn{15}{|c|}{$\cep{\nu}=0.1\qquad$ ({\tt maxrank}=50)}   \\ \hline 
     &    \multicolumn{5}{c|}{ \sscr\ }    & \multicolumn{5}{c|}{ \ssmr\ } &    \multicolumn{4}{c|}{ \sc{lr-pgmres} }  \\ 
 $n$ &   $k$ & rank & {\sc pcg} &  Res & Time   &  $k$ & rank & {\sc pcg}  & Res & Time &  $k$ & Mem & Res & Time  \\
\hline
1,024  & 4 &  37 & --    &  7.2e-8   &  0.7  &  
 4   &  37   & --  &  2.8e-7   &  0.6 & 
17  & 835  &  4.3e-7    &  2.9  \\

2,048 &   4  &  39 & --   &  2.0e-7  &   0.8   &  4  &  39  & --  &  2.7e-7   &  0.7 &  15  & 779  &   6.0e-7  &   4.0\\

4,096 &    3   &  38 & --    & 9.9e-7   & 0.5  &   
4 &   41  & --  &  1.6e-7   &  1.0 & 14  & 780  &  4.8e-7  &  7.4  \\ 

8,192 &    3 &  39 & --   & 5.6e-7   & 1.0 &
3  &   40 & --  & 4.6e-7   & 0.9 & 
13 & 775  & 6.7e-7  &  16.4 \\
16,384  &    3    &  40 & --     & 6.3e-7   & 1.8  & 
3  &  41  & --   & 6.4e-7   & 1.7 & 
12  & 757  & 6.5e-7 &  28.7   \\
 \hline 
      \multicolumn{15}{|c|}{$\cep{\nu}=0.01\qquad$ ({\tt maxrank}=70)}   \\ \hline 
     &    \multicolumn{5}{c|}{ \sscr\ }    & \multicolumn{5}{c|}{ \ssmr\ } &    \multicolumn{4}{c|}{ \sc{lr-pgmres} }  \\ 
 $n$ &   $k$ & rank & {\sc pcg} &  Res & Time   &  $k$ & rank & {\sc pcg}  & Res & Time &  $k$ & Mem & Res & Time  \\
\hline
 1,024    & 12    &  50 & [73,97]   &  8.0e-7   &  3.5     &  
 15      &  50   & [94,106]  &  7.1e-7   &  4.8 & 
*  & *  &  *    &  *  \\
2,048             &   12     &  51 & [77,97]   &  7.0e-7    &   4.2     &  15      &  51  & [91,121]  &  7.0e-7   &  4.2 & 
 *  & *  &   *  &   *\\

4,096             &    11      &  55 & [68,119]   & 8.7e-7   & 5.6  &   
13 &   53  & [82,100]  &  9.6e-7   &  6.5 & *  & *  &  *  &  *  \\ 

8,192        &    11 &  59 & [69,102]   & 5.5e-7   & 10.5 &
12  &   56 & [79,118]  & 7.0e-7   & 8.9 & 
* & *  & *  &  * \\

16,384       &    10    &  70 &   [79,100]   & 5.2e-7   & 17.5  & 
10  &  66  & [83,93]   & 8.3e-7   & 14.9 & 
*  & *  & * &  *   \\
\hline
\end{tabular}{\caption{Performance of \sscr, \ssmr, and {\sc{lr-pgmres}} for the {convection-diffusion problem} preconditioned by~\eqref{eq:precond} with $\texttt{tol}=10^{-6}$, $\mathtt{maxit}=50$, and  $\texttt{toltrank}=10^{-10}$. Two-sided sketching is required in all cases. `*' means the method did not converge to the prescribed tolerance within 50 iterations.
\label{Tab-ConvDiff}}}
\end{center}
\end{table}

We first focus on results obtained for $\cep{\nu}=0.1$. Applying the ADI approximation to the preconditioner \eqref{eq:precond} results in an iteration count that is almost independent of $n$, thanks to the dominance of the diffusion part of the operator. The proposed randomization-based computation of the residual matrix norm is reliable; all the values in the `Res' column are below $\mathtt{tol}=10^{-6}$. 
Our new methods require a low number of iterations to meet the chosen stopping condition. This, along with the moderate value of {\tt maxrank}, leads to a very effective solution procedure with solve times one order of magnitude quicker than for {\sc lr-pgmres}. The main disadvantage of the latter solver is its large storage demand which is not comparable to that of our short recurrence methods. When $\cep{\nu}=0.01$, all methods require more iterations to converge. While both \sscr\ and \ssmr\ still achieve competitive results, converging quickly, {\sc lr-pgmres} does not converge within 50 iterations.

{To further demonstrate the improved performance of the new matrix recurrences over classical vector methods for (\ref{eqn:Kron}), in Table \ref{tab:bicgstab} we report the performance of BiCGStab($\ell$) (with $\ell=2$) for \eqref{eq:convdiff} in Kronecker form. Vector methods allocate vectors of length $n^2$; because of this high memory requirement, we only consider the first three values of $n$. The inbuilt MATLAB function \texttt{bicgstabl.m} was used \cite{matlab}, with stopping tolerance $10^{-6}$. Incomplete LU preconditioning with threshold $10^{-4}$ was also employed, requiring storage for about 4 times the number of nonzeros\footnote{The entries of the coefficient matrix $\cal A$ were first reordered using {\tt symamd} to limit fill-in.} of $\cal A$.}

{In addition to the strong memory limitations, the results in Table \ref{tab:bicgstab} demonstrate that the vector method is extremely expensive compared to the matrix iterations, even for the smaller values of $n$ considered. It is also worth noting that the set-up cost of the preconditioner is significant; the incomplete LU factorisation requires around $9.5$ seconds for the case $n=1,024$, rising to approximately $230$ seconds for $n=4,096$. These  costs are not included in the timings reported in Table \ref{tab:bicgstab}.}

{Finally, we notice that GMRES($m$) was also tested, with $m=5$ and the same preconditioner; timings were not better, while in general the method requires more memory than BiCGStab(2). Results are not reported.}

\tabcolsep3pt 
\begin{table}[htb]
\centering
\begin{tabular}{|c|c|c| c|r||c|c|c| c|r|} \hline
$\cep{\nu}$ & $n$ & \# Iter & Res & Time & $\cep{\nu}$ & $n$ & \# Iter & Res & Time\\
\hline
0.1 & 1,024  &  33 & 4.8e-07 & 4.0 & 0.01& 1,024  &  17 & 4.7e-07 & 2.1\\
    & 2,048  &  59 & 7.0e-07  & 32.4 &    & 2,048  &  33 & 2.6e-07 & 17.6\\
    & 4,096  &  109 & 9.7e-07 & 2,316.1 &     & 4,096  &  79 & 2.9e-07 & 1,709.3\\ \hline
\end{tabular}
\caption{Performance of BiCGStab$(2$) for the convection-diffusion problem \eqref{eq:convdiff}, using ILU preconditioning  with threshold $10^{-4}$.  Reported CPU times are in seconds. \label{tab:bicgstab}}
\end{table}

\section{Application to a stochastic Galerkin mixed finite element problem}\label{sec:diffusion}

We now apply \ssmr\ and \sscr\ to a challenging class of matrix equations that arises when we apply a stochastic Galerkin mixed finite element method (SG-MFEM) to a system of PDEs with uncertain coefficients. Specifically, we consider an SG-MFEM discretization of the parametric Darcy flow problem \cite{BPS2012}, \cite{saddlepaper} that leads to a prototypical parametric saddle point problem.  After reformulating the discrete problem as a matrix equation, we employ the new solvers with a one-term preconditioner that renders the left coefficient matrices non-symmetric. 
Previous work on low-rank solvers for SGFEM matrix equations has focused on parametric PDE models that yield linear systems with SPD matrices (see  \cite{MultiRB, Kookjin} and references therein).

\subsection{Parametric Darcy Flow Problem}  Let $D \subset \mathbb{R}^{2}$ be a spatial domain with boundary $\partial D = \partial D_{D} \cup \partial D_{N}$  and define the parameter domain $\Gamma :=[-1,1]^{M}$.  We consider the following boundary value problem: find $\vec{u}: D \times \Gamma \to \mathbb{R}^{2}$ (a velocity field) and $p: D \times \Gamma \to \mathbb{R}$ (a pressure field)  that satisfy $\rho$-a.s. on $\Gamma$,
	\begin{equation}\label{GW-flow}
	\begin{aligned}
	\kappa(\mathbf{x}, \mathbf{y})^{-1} \vec{u} (\mathbf{x}, \mathbf{y}) + \nabla p(\mathbf{x}, \mathbf{y})  & =  0 &  \quad  \textrm{ in }  \quad D, \\
	 \nabla \cdot \vec{u}(\mathbf{x}, \mathbf{y})  & =  0 &   \quad   \textrm{ in }  \quad D , \\
	 p (\mathbf{x}, \mathbf{y})  & = g(\mathbf{x})  & \quad   \textrm{ on }  \partial D_{D},   \\
	 \vec{u}(\mathbf{x}, \mathbf{y}) \cdot \vec{n} & =  0 &  \quad    \textrm{ on } \partial D_{N}.
	\end{aligned}
	\end{equation}
Here, we assume that $\kappa^{-1}$ is a parameter-dependent function of the form
\begin{align}\label{kappa-def}
 \kappa (\mathbf{x}, \mathbf{y})^{-1} := \kappa_{0}(\mathbf{x}) + {\sum_{r=1}^{M}}\kappa_{r}(\mathbf{x}) y_{r}, \qquad \mathbf{x} \in D, \qquad \mathbf{y} \in \Gamma,
 \end{align}
and the parameters $y_{r}: =\xi_{r}(\omega)$ are images of independent uniform random variables $\xi_{r} \sim U(-1,1)$ with joint probability density $\rho = (1/2)^{M}$. This model arises when the reciprocal of the permeability coefficient is represented as a random field.  The chosen model \eqref{kappa-def} mimics the separable structure of a truncated Karhunen--Lo\`eve (KL) expansion \cite{Lord} where $\kappa_{0}= \mathbb{E}[\kappa^{-1}]$ and $\kappa_{r}:=\sqrt{\lambda_{r}} \phi_{r}$ where $(\lambda_{r}, \phi_{r})$ is an eigenpair of the chosen covariance operator.  Crucially, $\lambda_{r}\to 0$  as $r \to \infty$ at a rate that depends on the smoothness of the covariance and when $ \| \kappa_{r} \|_{\infty} \to 0$ rapidly as $r \to \infty$, we expect to be able to approximate the solution well in low rank format. To set up a well-posed weak formulation, the following assumption is needed.
	\begin{assumption}\label{A-ass} $\kappa^{-1} \in L^{\infty}(D \times \Gamma)$ and there exist $\kappa_{\min}$ and $\kappa_{\max}$ such that \\ 
$0 < \kappa_{\min} \le \kappa^{-1}(\mathbf{x} ,\mathbf{y}) \le \kappa_{max} < \infty$, a.e. in $D \times \Gamma.$
\end{assumption} 
We also make the following assumption about the parameter-free part.
\begin{assumption}\label{A0_ass} $\kappa_{0} \in L^{\infty}(D)$ and there exist $\kappa_{0,\min}$ and $\kappa_{0, \max}$ such that \\
$0 < \kappa_{0, \min} \le \kappa_{0}(\mathbf{x}) \le \kappa_{0,\max} < \infty$, a.e. in $D.$
\end{assumption} 
\noindent To ensure that Assumption \ref{A-ass} holds, we assume $\tau : =  \frac{1}{\kappa_{0, min}}\sum_{r=1}^{M} \| \kappa_{r} \|_{\infty} <1$.
 
\subsection{Stochastic Galerkin Approximation} 

Following  \cite{saddlepaper}, we apply stochastic Galerkin approximation using tensor product spaces. Briefly,  we look for approximations $\vec{u}_{h,q} \in \mathbf{V}_{h} \otimes S_{q}$ and $p_{h,q} \in W_{h} \otimes S_{q}$  that satisfy the associated weak form of \eqref{GW-flow} where  $\mathbf{V}_{h} \subset H_{0,N}(div;D)$ and $W_{h} \subset L^{2}(D)$ are an inf-sup stable pair of finite element spaces associated with a spatial mesh on $D$ and $S_{q} \subset L_{\rho}^{2}(\Gamma)$. 
In the experiments below, we use lowest-order rectangular Raviart--Thomas elements. On $\Gamma$, we employ \emph{global} polynomial approximation of total degree $\le q$.  In this case, $ n_{q} : = \textrm{dim} (S_{q}) = (M+q)!/M! q!$, where $M$ is the number of input parameters.

If we group all spatial unknowns for both solution fields $\vec{u}_{h,q}$ and $p_{h,q}$ per parametric degree of freedom, then the finite-dimensional weak problem can be written as
\vspace{-0.1in}
 \begin{equation}\label{mat-eq}
{ \left(\bmat{G}_{0} \otimes \bmat{A}_{0} + \sum_{r=1}^{M} \bmat{G}_{r} \otimes \bmat{A}_{r}\right)}
{x} = {g}_{0} \otimes {f} 
\qquad \Leftrightarrow \qquad
{\cal A} x = b,
  \end{equation}
  where ${\mathcal{A}}$ is symmetric and indefinite.  Using $S_{q} = \textrm{span}\{ \psi_{1}(\mathbf{y}), \ldots, \psi_{n_{q}}(\mathbf{y}) \}$, we have
$$[\bmat{G}_{r}]_{i,j} :=  \mathbb{E}\left[ y_{r} \psi_{i} \psi_{j} \right], \qquad i,j=1, \ldots, n_{q}, \quad r=0,1, \ldots, M,$$
(where $y_{0}:=1$) and we elect to work with an orthonormal Legendre basis so that $\mathbb{E}\left[\psi_{i} \psi_{j} \right] = \delta_{i,j}$ and $\bmat{G}_{0} = \bmat{I}$. The vector ${g}_{0} \in \mathbb{R}^{n_{q}}$ denotes the first column of $\bmat{G}_{0}$. The matrices $\bmat{G}_{r}$ are symmetric for $r \ge 1$ but are indefinite. The finite element matrices
	\begin{eqnarray*}
	\bmat{A}_{0} = \left( \begin{array}{cc} \bmat{K}_{0} & \bmat{B}^{T} \\ \bmat{B} & \bmat{0} \end{array} \right), \qquad \bmat{A}_{r} = \left( \begin{array}{cc} \bmat{K}_{r} & \bmat{0} \\ \bmat{0} & \bmat{0} \end{array} \right), \quad r \ge1,
	\end{eqnarray*}
\noindent  are symmetric and indefinite  with $\bmat{K}_{0}$ positive definite (due to Assumption \ref{A0_ass}) and $\bmat{B}^{T}$ has full column rank. 
We define the vector $ f = [{g}^{T}, {0}^{T}]^{T}  \in \mathbb{R}^{n_{h}}$ with $n_{h} = \textrm{dim}(\mathbf{V}_{h}) + \textrm{dim}(W_{h})$ where ${g}$ incorporates the non-zero Dirichlet boundary condition.  Due to the Kronecker structure, \eqref{mat-eq} can also be written as a $p=(M+1)$-term matrix equation
\begin{equation}\label{SGFEM_mat}
\bmat{A}_{0} \bmat{X} \bmat{G}_{0} + \bmat{A}_{1} \bmat{X} \bmat{G}_{1} + \cdots + \bmat{A}_{M} \bmat{X} \bmat{G}_{M} = {f} {g}_{0}^{T},
\end{equation}
with coefficient matrices that are {symmetric} and indefinite, with solution $\bmat{X} \in \mathbb{R}^{n_{h} \times n_{q}}$.

\subsection{Preconditioned Matrix Equation}

Following the discussion in section~\ref{sec:preconditioning}, it is natural to use a one-term preconditioner for \eqref{SGFEM_mat} based on the pair  $(\bmat{A}_0, \bmat{G}_{0})$ so that $ {\mathcal{P}}^{-1}(\bmat{R}_{k+1}) =   \bmat{A}_{0}^{-1} \bmat{R}_{k+1}$ (since $\bmat{G}_{0} = \bmat{I}$).  This strategy is equivalent to preconditioning the Kronecker system with the symmetric and indefinite matrix $\mathcal{P} = \bmat{I} \otimes \bmat{A}_{0}$. This is a `mean-based' preconditioner as $\bmat{A}_{0}$ only incorporates the leading part $\kappa_{0}$ of the uncertain input. Such preconditioners are successful when the variance of the input is low to moderate relative to the mean. ${\mathcal P}$ can also be viewed as a constraint preconditioner. This is easier to see if one reorders the degrees of freedom and rewrites the coefficient matrix ${\cal A}$ of the linear system in \eqref{mat-eq} and the preconditioner ${\cal P}$ as

\begin{eqnarray}\label{GWflow_saddle}
	{\cal A} = \left(\begin{array}{cc} \sum_{r=0}^{M} \bmat{G}_{r} \otimes \bmat{K}_{r} & \bmat{I} \otimes \bmat{B}^{\top} \\  & \\  \bmat{I} \otimes \bmat{B} & \bmat{0} \end{array} \right), \qquad {\cal P} = \left(\begin{array}{cc}  \bmat{I} \otimes \bmat{K}_{0} & \bmat{I} \otimes \bmat{B}^{\top} \\  & \\  \bmat{I} \otimes \bmat{B} & \bmat{0} \end{array} \right).
	\end{eqnarray}
If enough memory is available to store vectors of length $n_{h}n_{q}$, one may use {\sc{minres}} \cite{MINRES} as a solver with an SPD preconditioner, as in \cite{saddlepaper}.
However, for problems with the structure considered here, indefinite constraint preconditioners can be particularly effective. 

If properly initialized, {\sc{minres}} with a constraint preconditioner is equivalent to a projection method \cite{Gouldetal}.
Results in \cite{Lukvsan1998}, \cite{Kelleretal}, and  \cite{Rozloznik.Simoncini.02}, show that $1$ is an eigenvalue of the preconditioned system matrix with high multiplicity, 
and the remaining eigenvalues
are real and lie in the spectral interval of 
the SPD matrix $\sum_{r=0}^{M} \bmat{G}_{r} \otimes \bmat{K}_{r}$ preconditioned by $\bmat{I} \otimes \bmat{K}_{0}$.
Using Assumptions \ref{A-ass} and \ref{A0_ass}, one can show that this interval is contained in $[1-\tau, 1+ \tau] \subset \mathbb{R}^{+}$ so that all the eigenvalues are positive.

\subsection{Numerical Experiments}
We first apply \sscr\ and \ssmr\ to a synthetic problem where the coefficients in \eqref{kappa-def} decay rapidly, and $\bmat{X}$ can be approximated with $\texttt{maxrank} \le 40$. We then consider a more challenging case which requires a larger value of \texttt{maxrank} for the same tolerance \cep{(\texttt{tol})}. In both problems, an appropriate value of \texttt{maxrank} for a fixed value of $M$ is determined by running initial experiments on problems with a small value of $n_{h}$ (coarse spatial mesh). Since $\bmat{A}_{0} \in \mathbb{R}^{n_{h} \times n_{h}}$ becomes costly to invert (via  factorization) for fine spatial meshes, we apply an inexact preconditioner based on the pair $(\widetilde{\bmat{A}}_{0}, \bmat{I})$, where $\widetilde{\bmat{A}}_{0}$ is defined by replacing the (1,1) block of $\bmat{A}_{0}$ by the diagonal of $\bmat{K}_{0}$, which (since $\bmat{K}_{0}$ is a weighted mass matrix) is spectrally equivalent.  We fix $\bm{X}_{0}=0$,  $\texttt{tol}=10^{-6}$ and $\texttt{toltrank}=10^{-8}$ \cep{in all tests} and report the number of iterations $k$ required to meet the stopping condition,  the solution time (in seconds), and the actual rank of the final solution iterate $\bmat{X}_{k}$, as the number of parameters $M$ and the SG-MFEM discretization parameters $n_h$ and $q$ are increased. In Test Problem 2, where the dimensions of the reduced problems for ${\bm\alpha_{k}}$ and ${\bm\beta_{k}}$ are larger, we use {\sc{pcg}} for the inner solves \cep{and report the minimum and maximum number of  {\sc{pcg}} iterations}. 

\cep{In addition to the new subspace methods, we apply standard {\sc{minres}} to the associated linear systems with the same one-term preconditioner (denoted {\sc pminres}) and stopping condition. For Test Problem 1 (where memory requirements are lower and a comparison is feasible), we also apply the low-rank matrix-oriented {\sc{minres}} method from \cite{Stoll.Breiten.15}, again with the same one-term preconditioner (denoted {\sc lr-pminres}).}

\subsection*{Test Problem 1: Fast Decay Case}  Let $D = [0,1]^{2}$ with $p=g=1$ on  $\partial D_{D} =  \{0 \} \times [0,1]$ and $\vec{u} \cdot \vec{n} = 0$ on $\partial D_{N} = \partial D \setminus \partial D_{D}$, modelling flow from left to right across the domain. We choose $\kappa^{-1}$ as in \eqref{kappa-def} with $\kappa_{0}=1$ and $\kappa_{r}(\bm{x})= \sqrt{\lambda}_{r} \phi_{r}(\bm{x})$ where  $\sqrt{\lambda}_{r} = 0.832 r^{-4}$ and $\phi_{r}(\bm{x}) = \cos(2 \pi \beta_{1}(r) x_{1}) \cos(2 \pi \beta_{2}(r) x_{2})$ with $\bm{x} =[x_{1}, x_{2}]^{\top} \in D,$
where $\beta_{1}(r) = r - l(r)(l(r)+1)/2$, $\beta_{2}(r) = l(r) - \beta_{1}(r)$ and $l(r) = \lfloor - 1/2 + \sqrt{1/4 + 2r} \rceil$. This construction \cite{Eigel} provides a synthetic example of a KL expansion with rapidly decaying terms.
Choosing $M=5$ and $M=9$ (giving six and ten terms in the matrix equation) ensures we keep all terms with $\sqrt{\lambda_{r}} > 10^{-3}$ and $\sqrt{\lambda_{r}} >10^{-4}$. 

\tabcolsep3.2pt 
 \begin{table}[htb]
\begin{center}\cep{
\begin{tabular} {| c  c  r | c  c c  c  r |c  c c  c  r |  } \hline 
 \multicolumn{13}{|c|}{$ n_h=49,152$}   \\ \hline 
\multicolumn{3}{|c|}{} &    \multicolumn{5}{c|}{ \ssmr\ } &    \multicolumn{5}{c|}{ \sscr\ }   \\ 
 $M$ &   $q $  & $n_{q}$ & \texttt{maxrank} & Rank & $k$  & {Res} & Time   & \texttt{maxrank} & Rank & $k$  & {Res} & Time  \\ \hline
 &  4  &  126  & {40}  &  {35}  &  10  &  6.1e-07  &  11  & 30  &  30 &  10  &  5.5e-07 & 12  \\
5 &  5  & 252  & {40}  &  {40} &  12 &  7.0e-07  & 14  & 30  & 30 &  11  &  8.7e-07  & 14  \\
&  6  &  462  & 40  & 40  &  13   &  8.5e-07 & 16 & 40  & 40 &  11  &  5.2e-07   & 20 \\ \hline
& 4  & 715   &  40 &  40  &  11 & 8.4e-07  &  32  & 40 &  40  &  9 & 6.2e-07  & 38  \\
9  & 5 &  2,002  & 40  &  40  &  13  & 9.3e-07 &  \textbf{40}   & 40  &  40 &  10 & 8.0e-07  &  \textbf{47}  \\
 & 6 &  5,005  & 40  &  40 &  15   & 7.9e-07  & \textbf{47} & 40  &  40 &  11  & 7.7e-07  & \textbf{57}    \\ \hline 
\multicolumn{13}{|c|}{$ n_h=196,608$}   \\ \hline 
 \multicolumn{3}{|c}{} &    \multicolumn{5}{|c|}{ \ssmr\ }  &    \multicolumn{5}{c|}{ \sscr\ }    \\ 
 $M$ &   $q $  & $n_{q}$ & \texttt{maxrank} & Rank & $k$  & {Res} & Time   & \texttt{maxrank} & Rank & $k$  & {Res} & Time \\ \hline
   &  4  &  126   & 30  &  30   &   12  &  7.4e-07  &  44 & 30  &  30  &  9  &  9.0e-07 & 49   \\
5  &  5  &  252  &  {40}  & {39} &  12  &  7.8e-07  & 59 & 30  & 30   &  11 &  7.2e-07 & 69  \\
   &  6  &    462  & 40  & 40    &   13 &  9.0e-07  & 69   & 40  & 40  &  10 &  8.5e-07  & 77    \\ \hline
& 4  &   715  &  40 &  40  &  11  & 7.7e-07   &  {126} &  40 &  40  &  9  & 4.9e-07  &  {154} \\
9 & 5 &  2,002  & 40  &  40   &  13  & 8.0e-07  & \textbf{154}   & 40  &  40  & 10  &6.8e-07  & \textbf{184} \\
& 6 &  5,005  & 40  & 40   & 14  & 9.1e-07  & \textbf{164}   & 40  & 40   & 11  & 6.2e-07  & \textbf{204}  \\ \hline 
\end{tabular}{\caption{\textbf{Test Problem 1} \cep{Performance of \ssmr\ and \sscr\ on the matrix equation formulation with one-term preconditioning, $\texttt{tol}=10^{-6}$ and  $\texttt{toltrank}=10^{-8}$. Timings are rounded to the nearest second.}\label{TabMR1}}}}
\end{center}
\end{table}

\tabcolsep6.5pt 
\begin{table}[htb]
\begin{center}\cep{
\begin{tabular} {| c  c  r | c c r |  c c c   r | } \hline 
 \multicolumn{10}{|c|}{$ n_h=49,152$}   \\ \hline 
\multicolumn{3}{|c|}{} &       \multicolumn{3}{c|}{\sc{pminres}}  &     \multicolumn{4 }{c|}{\sc{lr-pminres}} \\ 
 $M$ &   $q $  & $n_{q}$ &  $k$   & {Res} & Time & \texttt{maxrank} &  $k$ &  {Res} & Time  \\ \hline
 &  4  &  126  &  18  &  8.6e-07  & 11  &  40 &  18  &  8.7e-07 &  147 \\
5 &  5  & 252 &   19  & 2.7e-07 &  23  & 40  & 20   &7.9e-07 & 161 \\
&  6  &  462  &    19  & 2.2e-07  & 48 & 40  & 21   & 8.1e-07 & 172 \\ \hline
& 4         & 715   &    18 & 8.6e-07  & 106 & 40  & 18   & 9.4e-07 & 170 \\
9  & 5 &  2,002   &   19  & 2.7e-07 &   575  & 45$\dagger$  & 20&  7.9e-07 & 237 \\
 & 6 &  5,005   & *  & *  &  *   &  45$\dagger$  & 21  & 6.2e-07 &  267   \\ \hline 
\multicolumn{10}{|c|}{$ n_h=196,608$}   \\ \hline 
 \multicolumn{3}{|c}{} &      \multicolumn{3}{c|}{\sc{pminres}}   &     \multicolumn{4}{c|}{\sc{lr-pminres}} \\ 
 $M$ &   $q $  & $n_{q}$  &  $k$ & {Res} & Time  &  \texttt{maxrank} & $k$  & {Res} & Time  \\ \hline
   &  4  &  126  & 18 &  8.6e-07   &  49 & 40$\dagger$  & 18  & 8.7e-07 & 606  \\
5  &  5  &  252  &  19 & 2.5e-07  &  133  & 40 & 19             &  4.6e-07 & 631 \\
   &  6  &    462    &  19 & 2.0e-07  &  390  & 40 & 20          & 8.3e-07 & 680    \\ \hline
& 4  &   715  &  18 &  8.6e-07  &  875  &40  & 18     &  9.2e-07 & 699  \\
9 & 5 &  2,002  &    *  &  * &  *  & 40  & 19   & 6.8e-07 & 716  \\
& 6 &  5,005  &  * &  * &  *     &  45$\dagger$& 20 &   7.6e-07 & 987  \\ \hline 
\end{tabular}{\caption{\textbf{Test Problem 1} \cep{Performance of {\sc{pminres}} and {\sc{lr-pminres}} with one-term constraint preconditioning and $\texttt{tol}=10^{-6}$. For {\sc{lr-pminres}}, $\texttt{toltrank}=10^{-8}$, and \texttt{maxrank} is chosen as for \ssmr\ in Table \ref{TabMR1} unless indicated (by $\dagger$). The symbol $*$ indicates that the experiment was aborted due to excessive memory/time requirements. Timings are rounded to the nearest second.}\label{TabCR1}}}}
\end{center}
\end{table}

 \cep{In Table \ref{TabMR1}},  we display results obtained with \ssmr\ and \sscr\ for problems discretized on two finite element meshes on $D$ and with polynomials of total degree $\le q=4,5,6$ on $\Gamma$. In most cases the final rank is equal to the chosen value of \texttt{maxrank}. We see that the number of iterations $k$ required by both subspace methods is independent of $M$ and the discretization parameters. \ssmr\ is generally quicker, \cep{however}, requiring only $1$--$4$ more iterations than \sscr.  \cep{Timings in bold indicate cases where two-sided sketching was applied in the estimation of the residual norm.  On the finest mesh, with  $M=9$ and $q=6$ so that $n_q=5,005$, the discrete problem consists of over 984 million equations. The problem is solved with the new subspace methods with modest memory requirements in a couple of minutes.} 

 \cep{In Table \ref{TabCR1},  we report results obtained with {\sc{pminres}} and {\sc{lr-pminres}}. The latter was designed in \cite{Stoll.Breiten.15} specifically for a three-term matrix equation associated with a PDE-constrained optimization problem, so some modification was needed. Since the stopping condition \vs{in {\sc{lr-pminres}}} is based on a quantity that can be significantly lower than the true residual \vs{norm}, to facilitate a fair comparison with our subspace methods, we applied {\sc{lr-pminres}} with the same value of \texttt{maxrank} and \texttt{toltrank} that was used 
 \vs{for} \ssmr\ 
 and then experimented with the stopping tolerance to determine the number of iterations $k$ required to achieve a relative residual norm less than $10^{-6}$. For this test problem, a stopping tolerance of $10^{-7}$ (i.e., one order of magnitude lower) sufficed to achieve a final approximation of comparable accuracy. However, in a few cases (indicated with $\dagger$) it was necessary to increase \texttt{maxrank} to achieve this. The number of iterations $k$ required by {\sc{pminres}} and {\sc{lr-pminres}} is similar, as expected, and independent of $M$ and the discretization parameters, thanks to the choice of preconditioner. For very small problems, {\sc{pminres}} is the quickest method. For larger problems, {\sc{lr-pminres}} is quicker than {\sc{pminres}}. However, comparing the results in Tables \ref{TabMR1} and \ref {TabCR1}, it is clear that the new subspace methods outperform both {\sc{pminres}} and {\sc{lr-pminres}} by a substantial margin in terms of timings.}

 \subsection*{Test Problem 2: Slow Decay Case}  Next, we consider $D=[-1,1]^{2}$ with $p=g=1$ on $\partial D_{D} =  \{-1 \} \times [-1,1]$ and $\vec{u} \cdot \vec{n} = 0$ on $\partial D_{N} = \partial D \setminus \partial D_{D}$.   This time, we model $\kappa^{-1}$ as a truncated KL expansion in terms of random variables $\xi_{r} \sim U(-\sqrt{3}, \sqrt{3})$ with mean $\kappa_{0}=1$ and separable exponential covariance
 $$C(\bm{x}, \bm{x}') = \sigma^{2} {\prod_{i=1}^{2} \exp\left(-\frac{ | x_i - x_i' |}{\ell_{i}}\right)}, \qquad \bm{x}, \bm{x}' \in D. $$
 In the parametric representation \eqref{kappa-def},  we then have $\kappa_{r} = \sigma \sqrt{3} \sqrt{\lambda}_{r}  \phi_{r}$ where $\{(\lambda_{r}, \phi_{r})\}$ are eigenpairs of $\sigma^{-2}C$, and these can be computed analytically \cite{Lord}.  For the full field, $\int_{D} \textrm{Var}[\kappa^{-1}] d \bm{x} = 4 \sigma^{2}$ and for the truncated one, we have $\sigma^{2} \sum_{r=1}^{M} \lambda_{r}$. This fact may be used to decide on an appropriate value for $M$. It is well known that the separable exponential covariance is problematic. The eigenvalues decay very slowly, especially for small correlation lengths $\ell_i$.  We include it to illustrate the limiting performance of our new methods on a difficult problem where the required rank is not that small, and to motivate the need for future work.

 \tabcolsep5.5pt 
 \begin{table}[ht!]
\begin{center}
\begin{tabular} {| c  c  r | c  c  c c  r | c c r |  } \hline 
 \multicolumn{11}{|c|}{$ n_h=49,152$}   \\ \hline 
\multicolumn{3}{|c|}{} &  \multicolumn{5}{c|}{ \ssmr }   &  \multicolumn{3}{c|}{ \sc{pminres} }  \\ 
 $M$ &   $q $  & $n_{q}$ & \texttt{maxrank}  & $k$ & {\sc pcg}  & Res & Time  &  $k$ & Res & Time  \\
\hline
& 4 &  495 & 80    &  7  & [6, 7]   &  7.5e-07   &  24   & 10   &   6.6e-07    &  32 \\
8  & 5 &  1,287  & 100   &  7  & [6, 7]   &  5.1e-07 &  30    & 10  &   5.2e-08  &  110 \\
 &  6 &  3,003  & 100   &  7  &  [6, 7]   &  7.0e-07 & \textbf{33} &   9  &   9.4e-07  & 381  \\ \hline
 & 4  & 1,820  &  120   & 7  &  [6, 7]  & 8.0e-07  & 75 & 10 &  3.4e-07  &234 \\
12 & 5 &  6,188  & 140  &  7  &  [6, 7]  & 8.0e-07 & \textbf{102}  & *  & * &  *  \\ 
& 6 & 18,564  & 140   & 8 &  [5 ,7]  & 5.1e-07 & \textbf{139}  & * &  * &  *       \\ \hline \hline
  \multicolumn{11}{|c|}{$ n_h=196,608$}   \\ \hline 
 \multicolumn{3}{|c|}{} &    \multicolumn{5}{c|}{ \ssmr\ }    &     \multicolumn{3}{c|}{ \sc{pminres} }  \\ 
 $M$ &   $q $  & $n_{q}$ & \texttt{maxrank} & $k$ & {\sc pcg}  & Res & Time   &  $k$ & Res & Time  \\
\hline
  &  4  &  495  & 80  &   7 & [6, 7]  &  6.8e-07   & 100 &    10  & 6.6e-07 & 252  \\
8  &  5  &  1,287 & 100  & 7  & [6, 7]  & 4.1e-07  & 165  &  10  & 5.4e-08 & 987  \\
&  6  &  3,003  & 100  &  7 &  [6, 7] & 5.1e-07   & \textbf{168} & * & * & *  \\ \hline
 & 4  & 1,820   &  120  &  7 &  [6, 7] & 6.8e-07   & 421  &   * &  *  & *  \\
12  & 5 &  6,188 & 140     & 7 &  [6, 7]    &  6.2e-07  &  \textbf{483} &   *  &  * &  *  \\
 & 6 & 18,564    & 140   &  7 & [6, 7] & 7.9e-07 & \textbf{491} &  * & * & *       \\ \hline 
\end{tabular}{\caption{\textbf{Test Problem 2} \cep{Performance of \ssmr\ on the matrix equation formulation with one-term preconditioner, $\texttt{tol}=10^{-6}$ and $\texttt{toltrank}=10^{-8}$, and {\sc{pminres}} on the associated linear system. The symbol $*$ indicates that the experiment was aborted due to excessive memory/time requirements. Timings are rounded to the nearest second.} \label{TabMR2}}}.
\end{center}
\end{table}

Recall, the computation of $\boldsymbol{\alpha}_{k}$ and $\boldsymbol{\beta}_{k}$ involves solving reduced problems with $(M+1)^{2}$ terms. This squaring, coupled with the not-so-small required values of \texttt{maxrank} (see Table \ref{TabMR2}) poses a computational challenge. 
In all cases, the rank of the final iterate was equal to the stated value of \texttt{maxrank}. Since we require larger values of \texttt{maxrank}, we confine our study here to a problem with $\ell_{1}=\ell_{2} = 2$ (large correlation length) so that we do not need to choose $M$ to be too large. We fix the standard deviation to be $\sigma=0.15$ so that the truncated field remains spatially positive and the mean-based preconditioner is effective.  Specifically, we consider $M=8$ and $M=12$, so that the associated matrix equations have nine and thirteen terms, and we retain $87\%$  and $89\%$ of the variance of the random input field, respectively.  

Results obtained with \ssmr\ are presented in Table \ref{TabMR2}. Timings in bold indicate cases where the problem dimension is large enough that two-sided sketching is needed. Results with \sscr\ are not shown. In most cases, it converged in one fewer iteration but  was substantially slower than \ssmr\ for larger problems due to the increased computational effort required to compute $\boldsymbol{\beta}_{k}$. Again, the number of iterations required by both methods is independent of the SG-MFEM discretisation parameters. The iteration counts are lower than in the last example. This is to be expected as the variance of the random input is smaller, making the mean-based preconditioner more effective. Due to the higher values of  \texttt{maxrank} required, the inner solves for the reduced problems had to be performed with {\sc pcg}, but inner iteration counts are also independent of the discretization parameters.  {\sc{pminres}} with the constraint preconditioner also converges well in terms of iteration counts. However, \cep{as expected}, it is not competitive in terms of time or memory requirements. On the finest spatial mesh, with  $M=12$ and $q=6$, the discrete problem consists of more than $3.9\times 10^{9}$ equations. \cep{With \ssmr\, we can solve the system in a few minutes with modest resources.}

\section{Conclusions}\label{sec:Conclusions}
We have derived a new class of {short} recurrences that can be used to solve linear multiterm matrix equations associated with general nonsymmetric coefficient operators, when the solution can be approximated by a low-rank matrix. The iterative methods are effective thanks
to a careful treatment of the inherent structure throughout the solution process. Rank truncations and randomization 
strategies are fundamental ingredients of our approach to keep memory allocations small and to facilitate the solution of very large problems.

The reported results also show that the new strategies are able to efficiently solve matrix equations in the target class for a wide range of parameter values and discretization settings. \vs{Our numerical experiments suggest that, when a good preconditioner is available, the \ssmr\ method shows better performance in terms of CPU time than \sscr, while ensuring lower memory requirements. However, it is certainly possible that for other problems, with other preconditioning strategies, \sscr\ might be more competitive. Here, we focused on introducing a general family of methods for solving nonsymmetric matrix equations which, as in the vector setting, may be enhanced by allowing a longer term recurrence. Working with longer recurrences clearly provides additional computational challenges. We leave this important open problem to future research.}

\section*{Acknowledgments}
\cep{The authors would like to thank Martin Stoll for providing the {\sc lr-minres} code that was used in section 8}, \vs{and the reviewers for their insightful comments}. {VS would also like to thank Maike Meier for pointing to \cite{MeierPhD.24} for bounds used in section~\ref{Randomized residual matrix}.}
The work of VS was partially supported by the
European Union - NextGenerationEU under the National Recovery and
Resilience Plan (PNRR) - Mission 4 Education and research
- Component 2 From research to business - 
Investment 1.1 Notice Prin 2022 - DD N. 104 of 2/2/2022,
entitled “Low-rank Structures and Numerical Methods in Matrix
and Tensor Computations and their
Application”, code 20227PCCKZ – CUP J53D23003620006. The same fund partially supported the visit of CP to the University of Bologna in September 2025. DP and VS are members of INdAM, Research Group GNCS. CP gratefully acknowledges the Dame Kathleen Ollerenshaw travel fund, administered by the University of Manchester. 

All authors acknowledge that they conducted some of this work at the Institute for Computational and Experimental Research in Mathematics (ICERM) in Providence, USA, which is supported by the National Science Foundation under Grant No. DMS-1929284, while participating in the Stochastic and Randomized Algorithms program, Spring semester 2026.
\bibliographystyle{siamplain}
\bibliography{references}

\end{document}